\documentclass[11pt,leqno]{article}
\usepackage{amsthm,amsfonts,amssymb,amsmath,oldgerm}
\usepackage{epsfig}
\numberwithin{equation}{section}
\usepackage[thinlines]{easybmat}

\newcommand{\pv}{{\rm P.V.}}
\newcommand{\DDD}{D3'}

\newcommand{\calE}{\mathcal{E}}
\newcommand{\bu}{\bar{U}}
\newcommand{\bU}{\bar{U}}

\def\eps{\varepsilon }

\newcommand\R{\mathbb R}

\def\eps{\varepsilon}

\newcommand\errfn{\textrm{errfn}}
\newcommand\br{\begin{remark}}
\newcommand\er{\end{remark}}
\newcommand\bp{\begin{pmatrix}}
\newcommand\ep{\end{pmatrix}}
\newcommand\be{\begin{equation}}
\newcommand\ee{\end{equation}}
\newcommand\ba{\begin{equation}\begin{aligned}}
\newcommand\ea{\end{aligned}\end{equation}}

\newcommand{\bap}{\begin{app}}
\newcommand{\eap}{\end{app}}
\newcommand{\begs}{\begin{exams}}
\newcommand{\eegs}{\end{exams}}
\newcommand{\beg}{\begin{example}}
\newcommand{\eeg}{\end{exaplem}}
\newcommand{\bpr}{\begin{proposition}}
\newcommand{\epr}{\end{proposition}}
\newcommand{\bt}{\begin{theorem}}
\newcommand{\et}{\end{theorem}}
\newcommand{\bc}{\begin{corollary}}
\newcommand{\ec}{\end{corollary}}
\newcommand{\bl}{\begin{lemma}}
\newcommand{\el}{\end{lemma}}
\newcommand{\bd}{\begin{definition}}
\newcommand{\ed}{\end{definition}}
\newcommand{\brs}{\begin{remarks}}
\newcommand{\ers}{\end{remarks}}

\newtheorem{theo}{Theorem}[section]
\newtheorem{prop}[theo]{Proposition}
\newtheorem{cor}[theo]{Corollary}
\newtheorem{lem}[theo]{Lemma}

\newtheorem{exams}[theo]{Examples}

\numberwithin{equation}{section}

\newcommand{\CalE}{\mathcal{E}}

\newcommand{\RR}{{\mathbb R}}

\newcommand{\const}{\text{\rm constant}}

\newtheorem{theorem}{Theorem}[section]
\newtheorem{proposition}[theorem]{Proposition}
\newtheorem{corollary}[theorem]{Corollary}
\newtheorem{lemma}[theorem]{Lemma}
\newtheorem{definition}[theorem]{Definition}

\newtheorem{example}[theorem]{Example}
\newtheorem{remark}[theorem]{Remark}

\newcommand{\RM}{\mathbb{R}}

\newcommand{\CM}{\mathbb{C}}

\title{
Nonlinear stability of viscous roll waves}

\author{\sc \small
Mathew A. Johnson\thanks{Indiana University, Bloomington, IN 47405;
matjohn@indiana.edu: Research of M.J. was partially supported by an NSF Postdoctoral Fellowship under NSF grant DMS-0902192.}
~~~
Kevin Zumbrun\thanks{Indiana University, Bloomington, IN 47405;
kzumbrun@indiana.edu:
Research of K.Z. was partially supported
under NSF grants no. DMS-0300487 and DMS-0801745.}
~~~
Pascal Noble\thanks{Universit\'e Lyon I, Villeurbanne, France;
noble@math.univ-lyon1.fr:
Research of P.N. was partially supported by the French ANR Project no.
ANR-09-JCJC-0103-01.}
}
\begin{document}

\maketitle


\begin{center}
\end{center}

\begin{center}
{\bf 2000 MR Subject Classification}: 35B35.
\end{center}


\begin{abstract}
Extending results of Oh--Zumbrun and Johnson--Zumbrun for
parabolic conservation laws,
we show that spectral stability implies nonlinear stability
for spatially periodic viscous roll wave solutions of the one-dimensional
St. Venant equations for shallow water flow down an inclined ramp.
The main new issues to be overcome are incomplete parabolicity
and the nonconservative form of the equations, which leads to
undifferentiated quadratic source terms that cannot be handled using
the estimates of the conservative case.
The first is
resolved
by treating the equations in the more favorable
Lagrangian coordinates, for which one can obtain large-amplitude
nonlinear damping estimates similar to those carried out
by Mascia--Zumbrun in the related shock wave case, assuming only
symmetrizability of the hyperbolic part.
The second is
resolved
by the observation that, similarly as
in the relaxation and detonation cases,
sources occurring in nonconservative components experience
greater than expected decay,
comparable to that experienced by a differentiated source.
\end{abstract}


\bigbreak

\section{Introduction }\label{intro}

Roll waves are a well-known hydrodynamic instability
occurring in shallow water flow down an inclined ramp,
generated by competition between gravitational force and
friction along the bottom.
These can be modeled as periodic traveling-wave solutions
of the St. Venant equations for shallow water flow,
which take the form of hyperbolic or parabolic balance laws;
see \cite{D,N1,N2} for detailed discussions of existence in
the inviscid and viscous case.

The spectral and linear stability of roll waves has been studied
for the  inviscid St. Venant equations in \cite{N1} and
the viscous St. Venant equations in \cite{N2}.
However, up to now, the relation between
spectral, linearized, and nonlinear stability has remained an outstanding open
question.
%
In this paper, extending recent results of \cite{OZ4,JZ3,JZ4}
in the related conservation law case, we settle
this question by showing that
{\it spectral implies linearized and nonlinear stability.}

This opens the way to rigorous numerical and analytical
exploration of stability of roll waves and related
phenomena via the associated eigenvalue ODE,
a standard and numerically and analytically well-conditioned problem.
At the same time, it gives a particularly interesting application
of the techniques of \cite{OZ4,JZ3,JZ4}.
For, roll waves, by numerical and experimental observation, appear
likely to be stable, at least in some regimes.
In the conservation law case, by contrast,
periodic waves so far appear typically to be unstable \cite{OZ1}.

\subsection{Equations and assumptions}\label{s:equations}
Consider the one-dimensional
St. Venant equations approximating shallow water flow on an inclined ramp:
\ba \label{eqn:1Econslaw}
h_t + (hu)_x&= 0,\\
(hu)_t+ (h^2/2F+ hu^2)_x&= h- u^2 +\nu (hu_x)_x ,
\ea
where $h$ represents height of the fluid, $u$ the
velocity average with respect to height,
$F$ is the Froude number, which here is the
square of the ratio between speed of the fluid and speed of gravity waves,
$\nu={\rm Re}^{-1}$ is a nondimensional viscosity
equal to the inverse of the Reynolds number,
the term $u^2$ models turbulent friction along the bottom,
and the coordinate $x$ measures longitudinal distance along the ramp.

In Lagrangian coordinates, these appear as
\ba \label{eqn:1conslaw}
\tau_t - u_x&= 0,\\
u_t+ ((2F)^{-1}\tau^{-2})_x&=
1- \tau u^2 +\nu (\tau^{-2}u_x)_x ,
\ea
where $\tau:=h^{-1}$ and $x$ now denotes a Lagrangian marker
rather than physical location.
We will work with this form of the equations, as it is more convenient
for our analysis in several ways.
(Indeed, for the large-amplitude damping estimates
of Section \ref{damping}, it appears to be essential in order
to obtain quantitative bounds on amplitude; see Remark \ref{damprmk}.)

Denoting $U:=(\tau, u)$, consider a spatially periodic
traveling-wave solution
\be\label{eqn:tw}
U=\bar{U}(x -ct),
\ee
of \eqref{eqn:1conslaw}
of period $X$ 
and wavespeed $c$
satisfying the traveling-wave ODE
\be \label{e:second_order}
\begin{aligned}
-c\tau' - u'&= 0,\\
-c u'+ ((2F)^{-1}\tau^{-2})'&=
1- \tau u^2 +\nu (\tau^{-2}u')' ,
\end{aligned}
\ee

Integrating the first equation of \eqref{e:second_order}
and solving for $u= u(\tau):= q-c\tau$,
where $q$ is the resulting
constant of integration, we obtain a second-order scalar profile equation
in $\tau$ alone:
\be \label{e:profile}
c^2 \tau'+ ((2F)^{-1}\tau^{-2})'=
1- \tau (q-c\tau)^2 -c\nu (\tau^{-2}\tau ')' .
\ee
Note that nontrivial periodic solutions of speed $c=0$
do not exist in Lagrangian coordinates,
as this would imply $u\equiv q$, and
\eqref{e:profile} would reduce to a scalar first-order equation
\be \label{e:zerocprofile}
  \tau'= F\tau^3(\tau q^2-1) ,
\ee
which since it is scalar first-order has no nontrivial periodic solutions,
even degenerate ones (e.g., homoclinic or heteroclinic cycles)
that might arise in the singular $c\to 0$ limit.
Rather, there appears to be a Hopf bifurcation as $c$ approaches some
minimum speed for which periodics exist;  see \cite{N2}, Section 4.1
and Fig. 1, Section 4.2.3.

It follows then that periodic solutions of \eqref{e:profile} correspond to values
$(X,c,q,b)\in \RR^5$, where $X$, $c$, and $q$ denote period,
speed, and constant of integration, and $b=(b_1,b_2)$ denotes
the values of $(\tau,\tau')$ at $x=0$, such that
the values of $(\tau,\tau')$ at $x=X$ of the solution of
\eqref{e:profile} are equal to the initial values $(b_1,b_2)$.

Following \cite{Se1,OZ3,OZ4,JZ3,JZ4}, we assume:

(H1) $\bar \tau>0$, so that all terms in \eqref{eqn:1conslaw}
are $C^{K+1}$, $K\ge 3$.

(H2) The map $H: \,
\R^5  \rightarrow \R^2$	
taking $(X,c,q,b) \mapsto (\tau,\tau')(X,c,b; X)-b$
is full rank at $(\bar{X},\bar c, \bar b)$,
where $(\tau,\tau')(\cdot;\cdot)$ is the solution operator of \eqref{e:profile}.

By the Implicit Function Theorem,
conditions (H1)--(H2) imply that the set of periodic solutions
in the vicinity of $\bar U$ form a
smooth $3$-dimensional manifold
\be\label{manifold}
\{\bar U^\beta(x-\alpha-c(\beta)t)\},
\;
\hbox{\rm with $\alpha\in \RR$, $\beta\in \RR^{2}$}.
\ee

\br\label{H2rmk}
\textup{
The transversality condition (H2) could be replaced by
the more general assumption that
the set of periodic solutions
in the vicinity of $\bar U$ form a
smooth $3$-dimensional manifold \eqref{manifold}.
However, it is readily seen in this context that
(H2) is then implied by the spectral stability condition (D3)
of Section \ref{bloch}; that is, transversality is
necessary for our notion of spectral, or Evans, stability.
This situation is reminsiscent of that of the viscous shock
case; see, for example, \cite[S\ 1.2.3]{ZH}, or \cite{MaZ3,Z1}.
}
\er

\br\label{relaxrmk}
\textup{
Note that \eqref{eqn:1conslaw} is of $2\times 2$ viscous relaxation type
\be\label{relax}
U_t+f(U)_x-\nu (B(U)U_x)_x=\bp0\\q(U)\ep,
\quad q(U)=1-\tau u^2,
\ee
where $q_u=-2u \tau<0$ for solutions $u>0$ progressing down the ramp.
Thus, constant solutions are stable so long as the subcharacteristic
condition $\Big|\frac{u^3}{2}\Big|< \Big|\frac{u^3}{\sqrt{F}}\Big|$,
is satisfied, or $F<4$.
When the subcharacteristic condition is violated,
roll waves appear through Hopf bifurcation as parameters are varied
through the minimum speed $c_{\rm min}= \frac{1}{\sqrt{F \tau_0^3}}$;
see Appendix \ref{s:hopf}.
For $\nu=0$, violation of the subcharacteristic condition is
associated with subshocks and the appearance of discontinuous
roll waves observed by Dressler \cite{D};
see \cite{JK} for related, more general, discussion.
}
\er

\br\label{singrmk}
\textup{
The limit $\nu\to 0$ represents an interesting singular
perturbation problem in which the structure of the profile equations
simplifies, decoupling into fast and slow scalar components,
and converging to inviscid Dressler waves \cite{D,N1} in an
appropriate regime \cite{N2}.
%
This would be an interesting setting in which to investigate the
associated spectral stability problem.
Another interesting limit is Hopf bifurcation from the constant solution occurring at minimum speed of existence \cite{N2},
treated here in Section \ref{s:hopf}; see Remark \ref{instabrmk}.
}
\er

\subsubsection{Linearized equations}\label{evans}
Making the change of variables $x\to x-ct$ to
co-moving coordinates, we convert \eqref{eqn:1conslaw} to
\ba \label{eqn:co1conslaw}
\tau_t-c\tau_x - u_x&= 0,\\
u_t-cu_x + ((2F)^{-1}\tau^{-2})_x&=
1- \tau u^2 +\nu (\tau^{-2}u_x)_x ,
\ea
and the traveling-wave solution
to a stationary solution $U=\bar U(x)$
convenient for stability analyis.

Writing \eqref{eqn:co1conslaw} in abstract form
\be\label{ab}
U_t +f(U)_x=(B(U)U_x)_x +g(U)
\ee
and linearizing (\ref{eqn:co1conslaw}) about $\bar{U}(\cdot)$, we obtain
\be \label{e:lin}
v_t = Lv := (\partial_x B\partial_x   -\partial_x A +C) v,
\ee
where the coefficients
\ba\label{coeffs}
A&:= df(\bu) - (dB(\bu) (\cdot) )\bu_x
=\bp -c & -1\\ -\bar \tau^{_-3}(F^{-1}- 2\nu \bar u_x)&-c\ep ,\\
B&:=B(\bu)= \bp 0&0\\0 & \nu \bar \tau^{-2}\ep,
\quad C:=dg(\bar U)=\bp 0 & 0\\-\bar u^2& -2\bar u\bar \tau\ep
\ea
are periodic functions of $x$.  As the underlying solution $\bar{U}$ depends on $x$ only,
equation \eqref{e:lin} is clearly autonomous in time.
By separation of variables, therefore,
decomposing solutions into the sum
of solutions of form $v(x,t)=e^{\lambda t}v(x)$,
where $v$ satisfies the eigenvalue equation $(L-\lambda)v=0$,
or, equivalently, by taking the Laplace transform,
we may reduce the study of stability of $\bar U$ to
the study of the spectral properties of the linearized operator $L$.

As the coefficients of $L$ are $X$-periodic,
Floquet theory
implies that
%
its spectrum is purely continuous.
Moreover, its spectral properties may be conveniently analyzed by
Bloch decomposition, an analog for periodic-coefficient operators
of the Fourier decomposition of a constant-coefficient operator,
as we now describe.


\subsubsection{Bloch decomposition
and stability conditions}\label{bloch}

Following \cite{G,S1,S2,S3}, we define the family of operators
\be \label{e:Lxi}
L_{\xi} = e^{-i \xi x} L e^{i \xi x}
= (\partial_x+i\xi) B(\partial_x+i\xi)
-(\partial_x+i\xi) A +C
\ee
operating on the class of $L^2$ periodic functions on $[0,X]$;
the $(L^2)$ spectrum
of $L$ is equal to the union of the
spectra of all $L_{\xi}$ with $\xi$ real with associated
eigenfunctions
\be
w(x, \xi,\lambda) := e^{i \xi x} q(x, \xi, \lambda),
\label{e:efunction}
\ee
where $q$, periodic, is an eigenfunction of $L_{\xi}$.
By standard considerations \cite{N2},\footnote{
For example, the characterization \cite{G} of spectra
as the zero set of an associated Evans function.}
the spectra of $L_{\xi}$
consist of the union of countably many continuous
surfaces $\lambda_j(\xi)$.

Without loss of generality taking $X=1$,
recall now the {\it Bloch representation}
\be\label{Bloch}
u(x)=
\Big(\frac{1}{2\pi }\Big) \int_{-\pi}^{\pi}
e^{i\xi\cdot x}\hat u(\xi, x) d\xi
\ee
of an $L^2$ function $u$, where
$\hat u(\xi, x):=\sum_k e^{2\pi ikx}\hat u(\xi+ 2\pi k)$
are periodic functions of period $X=1$, $\hat u(\cdot)$
denoting with slight abuse of notation the Fourier transform of $u$
in $x$.
By Parseval's identity, the Bloch transform
$u(x)\to \hat u(\xi, x)$ is an isometry in $L^2$:
\be\label{iso}
\|u\|_{L^2(x)}=
\|\hat u\|_{L^2(\xi; L^2(x))},
\ee
where $L^2(x)$ is taken on $[0,1]$ and $L^2(\xi)$ on $[-\pi,\pi]$.
Moreover, it diagonalizes the periodic-coefficient operator $L$,
yielding the {\it inverse Bloch transform representation}
\be\label{IBFT}
e^{Lt}u_0=
\Big(\frac{1}{2\pi }\Big) \int_{-\pi}^{\pi}
e^{i\xi \cdot x}e^{L_\xi t}\hat u_0(\xi, x)
d\xi\
\ee
relating behavior of the linearized system to
that of the diagonal operators $L_\xi$.

Following \cite{JZ4}, we assume along with (H1)--(H2) the
{\it strong spectral stability} conditions:

(D1) $\sigma(L_\xi) \subset \{ \hbox{\rm Re} \lambda <0 \} $ for $\xi\ne 0$.

(D2) $\hbox{\rm Re} \sigma(L_{\xi}) \le -\theta |\xi|^2$, $\theta>0$,
for $\xi\in \R$ and $|\xi|$ sufficiently small.

(\DDD) $\lambda=0$ is an eigenvalue
of $L_{0}$ of multiplicity $2$.\footnote{
The zero eigenspace of $L_0$,
corresponding to variations along the $3$-dimensional manifold
of periodic solutions in directions for which period does
not change \cite{Se1,JZ4}, is at least $2$-dimensional
by linearized existence theory and (H2).
}

As shown in \cite{N2}, (H1)-(H2) and (D1)--(\DDD)
imply that there exist $2$ smooth eigenvalues
\be\label{e:surfaces}
\lambda_j(\xi)= -i a_j \xi +o(|\xi|)
\ee
of $L_\xi$ bifurcating from $\lambda=0$ at $\xi=0$;
see Lemma \ref{blochfacts} below.

Loosely following \cite{JZ4}, we make the further nondegeneracy hypotheses:

(H3) The coefficients $ a_j$ in \eqref{e:surfaces} are distinct.

(H4) The eigenvalue $0$ of $L_0$ is nonsemisimple, i.e., $\dim
\ker L_0=1$.

\noindent
The coefficients $a_j$ may be seen to be the characteristics
of an associated Whitham averaged system
\ba\label{whit}
M(\beta)_t + G(\beta)_x&=0,\\
\Omega(\beta)_t + (c(\beta) \Omega(\beta))_x&=0
\ea
linearized about the values of $M$, $G$, $c$, $\Omega$ associated
with the background wave $\bar u$,
where $M$ is the mean of $\tau$ over one period and
$F$ the mean in the $\tau$-coordinate
of a certain associated flux, $c$ is wave-speed,
and $\Omega$ frequency of nearby periodic solutions,
indexed as in \eqref{manifold} by $\beta\in \R^2$;
see \cite{N2,OZ3,OZ4}.\footnote{
Here, we follow the formalism and notation of \cite{OZ3,OZ4}.
}
System \eqref{whit} formally governs slowly modulated solutions
\be\label{formod}
\tilde u(x,t)=\bar u^{\beta(\eps x,\eps t)}(\Psi(x,t))+O(\eps),\qquad \eps\to 0
\ee
presumed to describe large spatio-temporal behavior $x$, $t\gg 1$,
where $\bar u^\beta(\cdot)$
as in \eqref{manifold} parametrizes the set of nearby periodic solutions,
$\Omega=\Psi_x$,  and $c=-\Psi_t/ \Psi_x$.

Thus, (D1) implies weak hyperbolicity
of the Whitham averaged system \eqref{whit} (reality of $a_j$),
while (H3) corresponds to strict hyperbolicity.
Condition (H4) holds generically, and corresponds to the assumption
that speed $c$ is nonstationary along the manifold of nearby
stationary solutions; see Lemma \ref{blochfacts}.\footnote{
The case that (H4) is violated may be treated as in \cite{JZ3}.
}
Condition (D2) corresponds to ``diffusivity'' of
the large-time ($\sim$ small frequency) behavior of the linearized system,
and holds generically given (H1)--(H4), (D1), and (D3').\footnote{
This amounts to nonvanishing of $b_j$ in the Taylor series expansion
$\lambda_j(\xi)=-ia_j\xi- b_j\xi^2$
guaranteed by Lemma \ref{blochfacts} given (H1)--(H4), (D1), and (D3').}
Condition (\DDD) also holds generically, and can be verified by
an Evans function computation as described in \cite{N1}.
As discussed in \cite{OZ1,Se1,JZ3,JZ4},
conditions (D1)--(D3') are conservation law analogs of
the spectral assumptions
introduced by Schneider in the reaction-diffusion case \cite{S1,S2,S3}.

\subsection{Main result}

\begin{theo}\label{main}
Assuming (H1)--(H4) and (D1)--(\DDD),
let $\bar U=(\bar \tau, \bar u)$ be a traveling-wave solution
\eqref{eqn:tw} of \eqref{eqn:1conslaw} satisfying
the derivative condition
\be\label{froudebd}
\nu \bar u_x < F^{-1}.
\ee
Then, for some $C>0$ and $\psi \in W^{K,\infty}(x,t)$, where $K\geq 3$ is as in (H1)
\ba\label{eq:smallsest}
\|\tilde U-\bar U(\cdot -\psi-ct)\|_{L^p}(t)&\le
C(1+t)^{-\frac{1}{2}(1-1/p)}
\|\tilde U-\bar U\|_{L^1\cap H^K}|_{t=0},\\
\|\tilde U-\bar U(\cdot -\psi-ct)\|_{H^K}(t)&\le
C(1+t)^{-\frac{1}{4}}
\|\tilde U-\bar U\|_{L^1\cap H^K}|_{t=0},\\
\|(\psi_t,\psi_x)\|_{W^{K+1,p}}&\le
C(1+t)^{-\frac{1}{2}(1-1/p)}
\|\tilde U-\bar U\|_{L^1\cap H^K}|_{t=0},\\
\ea
and
\ba\label{eq:stab}
\|\tilde U-\bar U(\cdot-ct)\|_{ L^\infty}(t), \; \|\psi(t)\|_{L^\infty}&\le
C
\|\tilde U-\bar U\|_{L^1\cap H^K}|_{t=0}
\ea
for all $t\ge 0$, $p\ge 2$,
for solutions $\tilde U$ of \eqref{eqn:1conslaw} with
$\|\tilde U-\bar U\|_{L^1\cap H^K}|_{t=0}$ sufficiently small.
In particular, $\bar U$ is nonlinearly bounded
$L^1\cap H^K\to L^\infty$ stable.
\end{theo}

Theorem \ref{main} asserts  not only
bounded $L^1\cap H^K \to L^\infty$ stability, a very weak notion of stability,
but also asymptotic convergence of $\tilde U$ to the modulated wave
$\bar U(x-\psi(x,t))$.

\br\label{nonlinp}
\textup{
With further effort, it may be shown that the results of
Theorem \ref{main} extend to all $1\le p\le \infty$ using the pointwise
techniques of \cite{OZ2}; see discussion, \cite{JZ3,JZ4}.
}
\er

\br\label{froudermk1}
\textup{
The derivative condition \eqref{froudebd} is effectively an upper bound on
the amplitude of the periodic wave; see Remark \ref{froudermk}.
As discussed in Remark \ref{damprmk}, this is precisely the condition
that the first-order part of the linearized equations \eqref{e:lin}
be symmetric hyperbolic
(i.e., that $A$ in \eqref{coeffs} be symmetrizable),
and reflects a subtle competition between
hyperbolic and parabolic effects.
(The first-order part of the inviscid equations
is always symmetric--hyperbolic, corresponding to the equations
of isentropic gas dynamics with $\gamma$-law gas.)
It is satisfied when either wave amplitude or viscosity coefficient $\nu$
is sufficiently small.
It is not clear whether this condition may be relaxed.
}

\textup{
We note that condition \eqref{froudebd} is satisfied
for {\it all} roll-waves computed numerically in \cite{N2}.
For, in Eulerian coordinates, this condition translates to
$ h_x/h<(c\nu F)^{-1}$.
Examining Fig. 1 of \cite{N2},
a phase portrait in $(h,h')$ for $F=6$, $\nu=0.1$, and
$1.89<c<1.91$, we see that all periodic orbits appear to lie
beneath the line $h'/h=.68 $, whereas $(c\nu F)^{-1} \approx .88$.
}
\er

It is straightforward using the bounds of Corollary
\ref{greenbds} to show for ``zero-mass'', or derivative,
initial perturbations, that nonlinear decay rates
\eqref{eq:smallsest}--\eqref{eq:stab} improve by factor $(1+t)^{-1/2}$,
to the rates seen in the reaction-diffusion case \cite{S1,JZ5} for
general (undifferentiated) localized perturbations.
In particular, the perturbed wave $\tilde U$ then decays
asymptotically in $L^\infty$ to the background wave $\bar U$
with Gaussian rate $(1+t)^{-1/2}$ as in the reaction-diffusion case.
Likewise, under an unlocalized initial perturbation, or,
equivalently, the integral of a localized perturbation
the difference between $\tilde U$ and $\bar U$ may be expected
to blow up at rate $(1+t)^{1/2}$- this is indeed the linearized
behavior- and, barring special nonlinear structure, there
seems no reason why the difference between $\tilde U$ and the
modulation $\bar U(\cdot-\Psi)$ should not blow up as well:
at best it remains bounded.
In the reaction-diffusion case, for comparison,
results announced in \cite{SSSU} assert
that $\tilde U$ remains close to $\bar U$
even under unlocalized perturbations, and approaches
the modulated wave at rate $(1+t)^{-1/2}$ in $L^\infty$.
That is, the behavior in the conservation (balance) law case compared
to that in the reaction-diffusion case is, roughly speaking,
shifted by one derivative.\footnote{
At a purely technical level, this can be seen by the appearance of
a Jordan block in the zero eigenspace of $L_0$, introducing factor $\xi^{-1}$ in the
description of low-frequency behavior (Lemma \ref{blochfacts}).
Recall that a factor $i\xi$ corresponds roughly to differentiation in the
Bloch representation, through its relation to the Fourier transform.
In the reaction--diffusion case, the zero eigenspace of $L_0$ is simple,
and no such factor appears.
}

This reflects a fundamental difference between modulational behavior
in the present, conservation (or balance) law setting from that
of the reaction--diffusion case.
Namely, in the reaction-diffusion case, the Whitham averaged system
reduces to a single equation $\partial_t (\Omega)+ \partial_x (\Omega c)=0$,
or, equivalently,
\be\label{eq:rdmodel}
\Psi_t + c(\Psi_x)\Psi_x=0,
\ee
where $\Omega:=\Psi_x$ denotes frequency and $c:=-\frac{\Psi_t}
{\Psi_x}$ wave speed, and $c$ and $\Omega$ are related by the linearized
dispersion relation along the family of periodic orbits
(in the case considered by Schneider \cite{S1}, $c\equiv 0$).
On the other hand, the Whitham averaged equations \eqref{whit} in the present
case are a genuine $2\times 2$ first-order hyperbolic system\footnote{
In general, the dimension of the Whitham averaged system is
equal to the dimension of the manifold of nearby periodic solutions,
modulo translations \cite{JZ4}.
}
 in $\Psi_x$ and wave-speed $c$,
$c$ now considered as an independent parameter; that is,
they describe modulation of the perturbed wave in frequency
$\Psi_x$ and speed $c$, with phase shift $\Psi$ determined
indirectly by integration of $\Psi_x$.

Assuming heuristically (as justified
at the linearized, spectral, level
by the Bloch analysis of Section \ref{prep}),
that modulational behavior is governed by
a second-order regularization of the first-order Whitham averaged system,
we have the standard picture of behavior under localized perturbation
as consisting of modulations in $(\Psi_x, c)$ given by
a pair of approximate Gaussians
propagating outward with Whitham characteristic speeds $a_1$ and $a_2$,
hence an associated, much larger modulation in $\Psi$ determined
by integration in the $\Psi_x$ component, given
by a sum of approximate errorfunctions propagating with the same speeds.

Indeed,
this is exactly the description given in \eqref{eq:errfns}
of the principal part of the kernel $e(x,t;y)$
determining $\Psi$ through \eqref{psi}.
Likewise, the principal part of the Green function of the linearized
equations about $\bar U$ is $\bar U'(x)e(x,t;y)$, showing that linearized
behavior to lowest order indeed consists of a translation, or multiple
of $\bar U'(x)$, with amplitude
$$
\Psi(x,t)=\int e(x,t;y)(\tilde U(y,0)-\bar U(y,0))dy;
$$
see the description of the Green function in Corollary \ref{greenbds}.
The same considerations show that the rate of convergence of $\tilde U$
to the modulation $\bar U(\cdot -\Psi)$ cannot be improved (or,
in the case of a nonlocalized perturbation, recovered) by modulating
in additional parameters such as wave speed $c$ or etc.
For, as indicated by the above discussion, all such modulations
represent smaller contributions by factor $(1+t)^{-1/2}$ than
that of the phase $\Psi$, comparable rather to frequency $\Psi_x$,
and thus may be ignored in consideration of blow-up vs.stability.

This picture of modulational behavior as ``filtering'' by integration
along a certain direction of the hyperbolic--parabolic system derived
by Whitham averaging
seems quite interesting at a phenomenological level, and
a genuinely novel aspect of the conservation (balance) law case.
In particular, the $\Psi_x$ component direction along which the integration
is performed is in general
independent of either characteristic mode, so that the resulting
behavior is essentially different from that exemplified by
\eqref{eq:rdmodel} of a single scalar equation as in
the reaction--diffusion case.

\subsection{Discussion and open problems}\label{s:discussion}
The extension from the parabolic conservation law to the
present case involves a number of new technical issues associated with
lack of parabolicity and nonconservative form.
We overcome these difficulties by combining the
arguments of \cite{JZ3,JZ4}, \cite{N2} with those of
\cite{MaZ4,Z1,TZ1} (real viscosity) and
\cite{MaZ1} and \cite{LRTZ,TZ2} (relaxation and combustion systems
both involving nonconservative terms).

An interesting open problem is the rigorous justification
of spectral stability of
roll waves approaching the inviscid case in the singular zero
viscosity limit, extending results of \cite{N2}.
We hope to carry this out in future work.
For related asymptotic analysis,
see the study in \cite{Z2} of the inviscid limit for detonations.

Another interesting open problem is the
numerical investigation of spectral stability of large-amplitude
roll waves.
In particular, it is an interesting question
whether violation of the apparently technical
``amplitude condition'' \eqref{froudebd}
corresponds to actual physical phenomena/instability.
This is not inconceivable, as \eqref{froudebd}
is needed in our argument
not only for nonlinear iteration, but also for high-frequency
linearized bounds.  As the condition that the first-order
part of the equations be symmetric hyperbolic, it may well have
such significance-- however, this is not yet clear.

It is straightforward to extend our results
to the two-dimensional small-amplitude case, by
working in Eulerian coordinates and substituting for the present
large-amplitude damping estimate the simpler small-amplitude version
of \cite{MaZ2}; see \cite{JZ3,JZ4} for the multi-dimensional analysis
of periodic waves.
%
However, there is some evidence that roll waves develop transverse
instabilities in multi-dimensions \cite{N3}.
If so, this suggests the question whether such instability might
be connected with bifurcation to multiply periodic waves.
The extension of our stability analysis to the multiply periodic
case, as suggested in \cite{JZ3,JZ4},
would be another very interesting open problem.

\section{Spectral preparation}\label{prep}

We begin by a careful study of the Bloch perturbation expansion near $\xi=0$.

\begin{lemma}\label{blochfacts}
Assuming (H1)--(H4), (D1), and (\DDD), the eigenvalues
$\lambda_j(\xi)$ of $L_\xi$
are analytic functions and the Jordan structure of the zero
eigenspace of $L_0$ consists
of a $1$-dimensional kernel and a single Jordan chain of height $2$,
where the left kernel of $L_0$ is spanned by the constant
function $\tilde f\equiv (1,0)^T$, and $\bar u'$ spans the right
eigendirection lying at the base of the Jordan chain.
Moreover, for $|\xi|$ sufficiently small,
 there exist right and left eigenfunctions
$q_j(\xi, \cdot)$ and $\tilde q_j(\xi, \cdot)$
of $L_\xi$ associated with $\lambda_j$ of form
$q_j=\sum_{k=1}^2 \beta_{j,k} v_k$ and $\tilde q_j=\sum_{k=1}^2 \tilde \beta_{j,k} \tilde v_k$
where $\{v_j\}_{j=1}^2$ and $\{\tilde v_j\}_{j=1}^2$ are dual bases of the total
eigenspace of $L_\xi$ associated with sufficiently small eigenvalues,
analytic in $\xi$,
with $\tilde v_2(0)$ constant and $v_1(0) \equiv \bar u'(\cdot)$;
$\xi^{-1}\tilde \beta_{j,1}, \tilde \beta_{j,2}$
and $\xi \beta_{j,1}, \beta_{j,2}$
are analytic in $\xi$; and $\langle \tilde q_j,q_k\rangle= \delta_j^k$.
\end{lemma}

\br
\textup{
Notice that the results of Lemma \ref{blochfacts} are somewhat unexpected since, in general, eigenvalues
bifurcating from a non-trivial Jordan block typically do so in a nonanalytic fashion,
rather being expressed in a Puiseux series in fractional
powers of $\xi$.\footnote{
This is, however, consistent with the picture
of behavior as being approximately governed by a
first-order Whitham averaged system with eigenvalue perturbation expansions
agreeing to first-order with the associated linearized homogeneous dispersion relation
\cite{N2,OZ3,OZ4}.
}
The fact that analyticity prevails in our situation is a consequence of the very special structure
of the left and right generalized null-spaces of the unperturbed operator $L_0$, and the special
forms of the equations considered.
}
\er

\begin{proof}
Recall that $L_\xi$ has spectrum consisting of isolated eigenvalues
of finite multiplicity \cite{N2,G}.
Expanding
\be \label{Lpert}
L_\xi=L_0 + i\xi L^1- \xi^2L^2,
\ee
where, by \eqref{e:Lxi},
\ba \label{Ls}
L_0&=\partial_x B\partial_x - \partial_x A + C,
\quad
L^1= (B\partial_x + \partial_x B -A),
\quad
L^2= B,
\ea
consider the spectral
perturbation problem in $\xi$ about the eigenvalue $\lambda=0$ of $L_0$.

Because $0$ is an isolated eigenvalue of $L_0$, the associated total
right and left eigenprojections $P_0$ and $\tilde P_0$ perturb
analytically in $\xi$,
giving projection $P_\xi$ and $\tilde P_\xi$ \cite{K}.
These yield in standard fashion
(for example, by projecting appropriately chosen fixed subspaces)
locally analytic right and left bases $\{v_j\}$ and $\{\tilde v_j\}$
of the associated total eigenspaces given by
the range of $P_\xi$, $\tilde P_\xi$.

Defining $V=(v_1, v_{2})$ and
$\tilde V=(\tilde v_1,  \tilde v_{2})^*$, $*$ denoting
adjoint, we may convert the infinite-dimensional
perturbation problem \eqref{Lpert} into a $2\times 2$
matrix perturbation problem
\be\label{Mpert}
M_\xi=M_0+ i\xi M_1 - \xi^2 M_2+O(|\xi|^3),
\ee
where $M_\xi:= \left<\tilde V_\xi^*, L_\xi V_\xi\right>$ and $\left<\cdot,\cdot\right>$
denotes the standard $L^2(x)$ inner product on the finite interval $[0,X]$.
That is, the eigenvalues $\lambda_j(\xi)$
lying near $0$ of $L_\xi$ are the eigenvalues
of $M_\xi$, and the associated right and left eigenfunctions
of $L_\xi$ are
\be\label{vecrel}
f_j=V w_j  \;\hbox{\rm  and } \;
\tilde f_j=\tilde w_j \tilde V^* ,
\ee
where $w_j$ and $\tilde w_j$
are the associated right and left eigenvectors of $M_\xi$.

By assumption, $\lambda=0$ is a nonsemisimple eigenvalue
of $L_0$, so that $M_0$ is nilpotent but nonzero, possessing a
nontrivial associated Jordan chain.
Moreover, using the fact that $\left<(1,0)^T, C\right>=0$, where, again, $\left<\cdot,\cdot\right>$ represents
the $L^2(x)$ inner product over the finite domain $[0,X]$, the function $\tilde f\equiv (1,0)^T$
by direct computation lies
in the kernel of $L_0^*=\left(\partial_xB^*\partial_x+A^*\partial_x+C^*\right)$,
%
%
we have that the
two-dimensional zero eigenspace of $L_0$ is consists precisely
of a one-dimensional kernel and a single Jordan chain of height two.
Moreover, by translation-invariance (differentiate in $x$
the profile equation \eqref{e:profile}), we have $L_0\bar u'=0$,
so that $\bar u'$ lies in the right kernel of $L_0$.

Now, recall assumption (H2) that $H: \,
\R^5  \rightarrow \R^2$	
taking $(X,c,q,b) \mapsto (\tau,\tau')(X,c,b; X)-b$
is full rank at $(\bar{X},\bar c, \bar b)$,
where $(\tau,\tau')(\cdot;\cdot)$ is the solution operator of \eqref{e:profile}.
The fact that $\ker L_0$ is one-dimensional implies that
the restriction
$\check H$ taking
$
(b, q)  \mapsto u(X; b, c, q)-b
$
for fixed $(X,c)$
is also full rank, i.e., $H$ is full rank with respect to
the specific parameters $(X,c)$.
Applying the Implicit Function Theorem and counting dimensions,
we find that the set of periodic solutions, i.e., the inverse
image of zero under map $H$ local to $\bar u$
is a smooth three-dimensional manifold
$\{\bar u^\beta (x-\alpha-c(\beta )t)\}$,
with $\alpha\in \RR$, $\beta \in \RR^{2}$.
Moreover, two dimensions may be parametrized by $(X,c)$,
or without loss of generality $\beta  =(X,c)$.

Fixing $X$ and varying $c$, we find by differentiation of \eqref{e:profile}
that $f_*:=-\partial_s \bar U$ satisfies the generalized eigenfunction
equation
$$
L_0 f_*= \bar U'.
$$
Thus, $\bar U'$ spans the eigendirection lying at the base of the
Jordan chain, with the generalized zero-eigenfunction of $L_0$
corresponding to variations in speed along the manifold of periodic
solutions about $\bar U$.
Without loss of generality, therefore, we may take
$\tilde v_{2}$ to be constant at $\xi=0$, and
$v_{1} \equiv \bar U'$ at $\xi=0$.

Noting as in \cite{JZ3} the fact that, by \eqref{coeffs},
\ba\label{comp}
A\bar U_x= f(\bu)_x - (\partial_x B(\bu))\bar U_x
&= \partial_x( f(\bu)_x - B(\bU)\bU_x) +B(\bu)\partial_x \bu_x
\\
& = g(\bU) +B(\bu)\partial_x \bu_x ,
\ea
and so by $e_2g=0$, $\partial_x e_2=0$, we have
$$
\langle e_2,L^1 \bar U'\rangle=
\langle e_2,(\partial_{x} B+B\partial_x -A  )\bar U'\rangle=
\langle e_2, \partial_{x}B  \bar U' \rangle
\equiv 0
$$
for $e_2:=(0,1)$,
where $\langle \cdot, \cdot\rangle$
denotes $L^2(x)$ inner product on the interval $x\in [0,X]$,
we find under this normalization that \eqref{Mpert} has the special structure
\ba\label{Mstructure}
M_0=\bp
 0 & 1\\
 0 & 0\ep,
\qquad
M_1=\bp
 * & *\\
 0 & *\ep.
\ea

Now, rescaling \eqref{Mpert} as
\be\label{rescale}
\check M_\xi:= (i\xi)^{-1} S(\xi)M_\xi S(\xi)^{-1},
\ee
where
\be\label{S}
S:=\bp
 i \xi & 0\\
 0 & 1\\
\ep,
\ee
we obtain
\be\label{checkMpert}
\check M_\xi=
\check M_0 + i\xi\check M_1 + O(\xi^2),
\ee
where $\check M_j= \check M_j$ like the original
$M_j$ are constant
and the eigenvalues $m_j(\xi)$ of $\hat M_\xi$ are
$(i\xi)^{-1}\lambda_j(\xi)$.

As the eigenvalues $m_j$ of $\check M_\xi$ are continuous,
the eigenvalues $\lambda_j(\xi)=i\xi m_j$ are differentiable
at $\xi=0$ as asserted in the introduction.
Moreover, by (H3), the eigenvalues $\check\lambda_j(0)$
of $\check M_0$ are distinct, and so they perturb analytically
in $\xi$, as do the associated right and left eigenvectors
$z_j$ and $\tilde z_j$.
Undoing the rescaling \eqref{rescale},
and recalling \eqref{vecrel}, we obtain the result.
\end{proof}

\section{Linearized stability estimates}\label{linests}
By standard spectral perturbation theory \cite{K}, the total
eigenprojection $P(\xi)$ onto the eigenspace of $L_\xi$
associated with the eigenvalues $\lambda_j(\xi)$, $j=1,2$
described in the previous section
is well-defined and analytic in $\xi$ for $\xi$ sufficiently small,
since these (by discreteness of the spectra of $L_\xi$) are
separated at $\xi=0$ from the rest of the spectrum of $L_0$.
By (D2), there exists an $\eps>0$ such that $\Re\lambda_j(\xi)\leq-\theta|\xi|^2$ for $0<|\xi|<2\eps$.
With this choice of $\eps$, we introduce a smooth cutoff function $\phi(\xi)$ that
is identically one for $|\xi|\le \eps$ and identically
zero for $|\xi|\ge 2\eps$, $\eps>0$ sufficiently small,
we split the solution operator $S(t):=e^{Lt}$ into
a low-frequency part
\be\label{SI}
S^I(t)u_0:=
\Big(\frac{1}{2\pi }\Big) \int_{-\pi}^{\pi}
e^{i\xi \cdot x}
\phi(\xi)P(\xi) e^{L_\xi t}\hat u_0(\xi, x) d\xi
\ee
and the associated high-frequency part
\be\label{SII}
S^{II}(t)U_0:=
\Big(\frac{1}{2\pi }\Big) \int_{-\pi}^{\pi}
e^{i\xi \cdot x}
\big(I-\phi P(\xi)\big)
e^{L_\xi t}\hat U_0(\xi, x)
d\xi.
\ee

Our strategy is to treat the high- and low-frequency operators separately since,
as is standard, the low-frequency analysis is considerably more complicated
than the corresponding high-frequency analysis.
That being said,
we begin by deriving bounds on the solution operator at
high frequency.

\subsection{High-frequency bounds}\label{HF}
By boundedness of the resolvent on compact subdomains of the resolvent
set, equivalence (as the zero-set of an associated Evans function \cite{N2,OZ1})
of $H^1$ and $L^2$ spectrum,
Assumption (D2),
and the high-frequency estimates of Lemma \ref{resbd},
we have for $|\xi|$ bounded away from zero and waves satisfying the amplitude condition \eqref{froudebd} that
the resolvent $(\lambda-L_\xi)^{-1}$ is uniformly bounded
from $H^1\to H^1$ for $\Re \lambda=-\eta< \theta<0$,
whence, by Pr\"uss' Theorem \cite{Pr},
$\|e^{L_\xi t}f\|_{H^1}\le Ce^{-\theta t} \|f\|_{H^1}$.

For $|\xi|$ sufficiently small, on the other hand, $\phi\equiv 1$,
and $I-\phi(\xi)P=I-P=Q$, where $Q$ is the eigenprojection
of $L_\xi$ associated with eigenvalues complementary to $\lambda_j(\xi)$,
which by spectral separation of $\lambda_j(\xi)$ from the
remaining spectra of $L_\xi$, have real parts strictly less than zero.
Applying Pr\"uss' Theorem to the restriction of $L_\xi$
to the Hilbert space given by the range of $Q$, we find, likewise,
that
$\| e^{L_\xi t} (I-\phi(\xi)) f\|_{H^1}=
\|e^{L_\xi t}Qf\|_{H^1} \le Ce^{-\theta t} \|f\|_{H^1}$.

Combining these observations,
we have the exponential decay bound
$$
\|e^{L_\xi t}(I-\phi P(\xi))f\|_{H^1([0,X])}\le Ce^{-\theta t}\|f\|_{H^1([0,X])}$$
for $\theta>0$ as in (D2) and $C>0$, from which it follows
\ba\label{semigp}
\|e^{L_\xi t}(I-\phi P(\xi))\partial_{x}^l f\|_{H^{1}([0,X])}
&\le   C e^{-\theta t}\|f\|_{H^{l+1}([0,X])}\\
\ea
for $0\le l\le K$ ($K$ as in (H1)).
Together with (\ref{iso}), these give immediately the
following estimates.

\begin{proposition}[\cite{OZ4}]\label{p:hf}
Under assumptions (H1)--(H4), (D1)--(D2), and assuming the amplitude condition \eqref{froudebd} holds,
there exists constants $\theta$, $C>0$, such that for all
all $t>0$, $2\le p\le \infty$, $0\le l\le 2$, $0\le m\le 2$, we have the high-frequency estimates
\ba\label{SIIest}
\|S^{II}(t)\partial_x^l f\|_{L^2(x)}&\le
Ce^{-\theta t}\|f\|_{H^{l+1}(x)},\\
\|S^{II}(t) \partial_x^m f\|_{L^p(x)}&\le
C e^{-\theta t}\|f\|_{H^{m+2}(x)}.
\ea
\end{proposition}

\begin{proof}
For $m,l=0$,
the first inequalities follow immediately by (\ref{iso})
and \eqref{semigp}.
The second follows for $p=\infty$ by Sobolev embedding.
The result for general $2\le p\le \infty$ then follows by
$L^p$ interpolation.
A similar argument applies for $1\le l,  m\le 2$ by higher-derivative
versions of \eqref{semigp}, which follow in exactly the same way.
\end{proof}

\subsection{Low-frequency bounds}\label{LF}

As noted above, analysis of the solution operator at low frequency
is considerably more complicated than the
high-frequency bounds outlined above.
To aid in our analysis, we introduce the Green kernel
\be\label{GI}
G^I(x,t;y):=S^I(t)\delta_y(x)
\ee
associated with $S^I$, and the corresponding kernel
\be\label{GIxi}
[G^I_\xi(x,t;y)]:=\phi(\xi)P(\xi) e^{L_\xi t}[\delta_{y}(x)]
\ee
appearing within the Bloch representation
of $G^I$, where the brackets on $[G_\xi]$ and $[\delta_y]$
denote the periodic extensions of these functions onto the whole line.
Then, we have the following descriptions of $G^I$, $[G^I_\xi]$,
deriving from the
spectral expansion \eqref{e:surfaces} of $L_\xi$ near $\xi=0$.

\begin{proposition}[\cite{OZ4}]\label{kernels}
Under assumptions (H1)--(H4) and (D1)--(\DDD),
\ba\label{Gxi}
[G^I_\xi(x,t;y)]&= \phi(\xi)\sum_{j=1}^{2}e^{\lambda_j(\xi)t}
q_j(\xi,x)\tilde q_j(\xi, y)^*,\\
G^I(x,t;y)&=
\Big(\frac{1}{2\pi }\Big)  \int_{\R} e^{i\xi \cdot (x-y)}
[G^I_\xi(x,t;y)] d\xi \\
&=
\Big(\frac{1}{2\pi }\Big)  \int_{\R}
e^{i\xi \cdot (x-y)}
\phi(\xi)
\sum_{j=1}^{2}e^{\lambda_j(\xi)t} q_j(\xi,x)\tilde q_j(\xi, y)^*
d\xi,
\ea
where $*$ denotes matrix adjoint, or complex conjugate transpose,
$q_j(\xi,\cdot)$ and $\tilde q_j(\xi,\cdot)$
are right and left eigenfunctions of $L_\xi$ associated with eigenvalues
$\lambda_j(\xi)$ defined in \eqref{e:surfaces},
normalized so that $\langle \tilde q_j,q_j\rangle\equiv 1$.
\end{proposition}

\begin{proof}
Relation  (\ref{Gxi})(i) is immediate from the spectral decomposition
for $C^0$ semigroups at eigenvalues of finite multiplicity, and the fact that $\lambda_j$
are distinct for $|\xi|>0$ sufficiently small, by (H3).
Substituting (\ref{GI}) into (\ref{SI})
and computing
\be\label{comp1}
\widehat{\delta_y}(\xi,x)=
\sum_k e^{2\pi i kx}\widehat{\delta_y}(\xi + 2\pi k e_1)=
\sum_k e^{2\pi i kx}e^{-i\xi \cdot y-2\pi i ky}
= e^{-i\xi \cdot y}[\delta_{y}(x)],
\ee
where the second and third equalities follow from the fact that the Fourier transform of either the continuous or discrete
the delta-function is unity, 
we obtain
\ba\label{GIsub}
G^I(x,t;y)&=
\Big(\frac{1}{2\pi }\Big)  \int_{-\pi}^{\pi}
e^{i\xi \cdot x} \phi P(\xi) e^{L_\xi t} \widehat{\delta_y}(\xi,x)d\xi\\
\nonumber
&=
\Big(\frac{1}{2\pi }\Big)  \int_{-\pi}^{\pi}
e^{i\xi \cdot (x-y)}  \phi P(\xi)e^{L_\xi t} [\delta_{y}(x)] d\xi,
\ea
yielding (\ref{Gxi})(ii) by (\ref{GIxi})(i) and the fact that $\phi$
is supported on $[-\pi,\pi]$.
\end{proof}


We now state our main result for this section, which
uses
the spectral representation of $G^I$ and $[G^I_{\xi}]$
described in Proposition \ref{kernels} to
decompose
the low-frequency
Green kernel into a leading order piece
(corresponding to translational modulation) plus a faster decaying residual.
Underlying this decomposition is the fundamental relation
\be
G(x,t;y)=
\left(\frac{1}{2\pi }\right)\int_{-\pi}^{\pi}\int_{\R^{d-1}}
e^{i\xi \cdot (x-y)}[G_\xi(x_1,t;y_1)]d\xi,
\ee
which
serves as
%
the crux of the low-frequency analysis
both here and in \cite{OZ2,JZ3}.

\begin{proposition} \label{Gbds}
Under assumptions (H1)-(H4) and (D1)-(\DDD),
the low-frequency Green function $G^I(x,t;y)$ of \eqref{GI} decomposes as
$G^I=E+\tilde G^I$,
\be\label{E1}
E=\bar U'(x)e(x,t;y),
\ee
where, for some $C>0$, all $t>0$,
\ba\label{GIest}
\sup_{y}\|\tilde G^I(\cdot, t,;y) \|_{L^p(x)}
 &\le  C (1+t)^{-\frac{1}{2}(1-\frac{1}{p})}\\
\sup_{y}\|\partial_{y}^r \tilde G^I(\cdot, t,;y) \|_{L^p(x)},
\quad
\sup_{y}\|\partial_{t}^r \tilde G^I(\cdot, t,;y) \|_{L^p(x)}
 &\le  C (1+t)^{-\frac{1}{2}(1-\frac{1}{p})-\frac{1}{2}}\\
\sup_{y}\| \tilde G^I(\cdot, t,;y)(0,1)^T \|_{L^p(x)}
 &\le  C (1+t)^{-\frac{1}{2}(1-\frac{1}{p})-\frac{1}{2}}\\
\ea
for $p\ge 2$, $1\le r\le 2$,
\ba\label{ederest}
\sup_{y}\| \partial_x^j \partial_t^l
\partial_{y}^r e(\cdot, t,;y) \|_{L^p(x)}
 &\le  C (1+t)^{-\frac{1}{2}(1-\frac{1}{p})- \frac{(j+l)}{2}-\frac{1}{2}}\\
\ea
for $p\ge 2$, $0\le j, l$, $j+l\le K+1$, $1\le r\le 2$,
and
\ba\label{eest}
\sup_{y}\|\tilde \partial_x^j \partial_t^l e(\cdot, t,;y) \|_{L^p(x)}
 &\le  C (1+t)^{-\frac{1}{2}(1-\frac{1}{p})-\frac{(j+l)}{2}}\\
\ea
for $0\le j, l$, $j+l\le K+1$,
provided that $p\ge 2$ and $j+l\ge 1$ or $p=\infty$.
Moreover, $e(x,t;y)\equiv 0$ for $t\le 1$.
\end{proposition}

\br\label{newrmk}
\textup{
The crucial new observation in the nonconservative case treated
here is \eqref{GIest}(iii), which asserts that sources entering
in the nonconservative second coordinate of the linearized equations
experience decay equivalent to that of a differentiated source
entering in the first coordinate.
This is what allows us to treat non-divergence-form source terms
arising in the second equation of the eventual perturbation equations.
}
\er

\begin{proof}
Recalling \eqref{Gxi} and Lemma \ref{blochfacts}, we have
\ba\label{Gnew}
G^I(x,t;y)&=
\Big(\frac{1}{2\pi }\Big) \int_{\R}
e^{i\xi \cdot (x-y)}
\phi(\xi)
\sum_{j=1}^{2}e^{\lambda_j(\xi)t} q_j(\xi,x)\tilde q_j(\xi, y)^*
d\xi\\
&=
\Big(\frac{1}{2\pi }\Big) \int_{\R}
e^{i\xi \cdot (x-y)}
\phi(\xi)
\sum_{j,k,l=1}^{2}e^{\lambda_j(\xi)t} \beta_{j,k}v_k(\xi,x)
\tilde \beta_{j,l}\tilde v_l(\xi, y)^*
d\xi,
\ea
the fact that $\beta_{j,1}=O(\xi^{-1})$ suggests the $k=1$ terms (corresponding to translation) dominate the low-frequency
Green kernel.  With this motivation, we define
\ba\label{enew}
\tilde e(x,t;y)&=
\Big(\frac{1}{2\pi }\Big) \int_{\R}
e^{i\xi \cdot (x-y)}
\phi(\xi)
\sum_{j,l}e^{\lambda_j(\xi)t} \beta_{j,1}
\tilde \beta_{j,l}\tilde v_l(\xi, y)^* d\xi
\ea
so that
\ba\label{Gdiff}
G^I&(x,t;y) -\bar U'(x)\tilde e(x,t;y)=\\
&\Big(\frac{1}{2\pi }\Big) \int_{\R}
e^{i\xi \cdot (x-y)} \phi(\xi)
\sum_{j, k\ne 1, l} e^{\lambda_j(\xi)t}
\beta_{j,k} \tilde \beta_{j,l}v_k(\xi,x)\tilde v_l(\xi, y)^* d\xi\\
&\quad +
\Big(\frac{1}{2\pi }\Big) \int_{\R}
e^{i\xi \cdot (x-y)} \phi(\xi)
\sum_{j, l} e^{\lambda_j(\xi)t}
\beta_{j,1} \tilde \beta_{j,l}
\Big(v_1(\xi,x)-\bar U'(x)\Big)
\tilde v_l(\xi, y)^* d\xi,\\
\ea
where, by analyticity of $v_1$,
$v_1(\xi,x)-\bar U'(x)=O(|\xi|)$, and so, by Lemma \ref{blochfacts},
\be\label{crucial}
\beta_{j,1} \tilde \beta_{j,l}
\Big(v_1(\xi,x)-\bar U'(x)\Big)
\tilde v_l(\xi, y)^* =O(1)
\ee
and
\be\label{crucial2}
\beta_{j,2} \tilde \beta_{j,l}v_2(\xi,x)\tilde v_l(\xi, y)^*
=O(1).
\ee
Note further that $\tilde v_l\equiv (1,0)^T$ unless $l=1$,
in which case $\tilde \beta_{jl}=O(|\xi|)$ by Lemma \ref{blochfacts};
hence
\be\label{crucial3}
\partial_{y}\Big(\beta_{j,1} \tilde \beta_{j,l}
\Big(v_1(\xi,x)-\bar U'(x))
\tilde v_l(\xi, y)^* \Big) =O(|\xi|),
\ee
\be\label{newcrucial3}
\Big(\beta_{j,1} \tilde \beta_{j,l}
\Big(v_1(\xi,x)-\bar U'(x))
\tilde v_l(\xi, y)^* \Big)(0,1)^T =O(|\xi|),
\ee
and
\be\label{crucial4}
\partial_{y}
\Big(\beta_{j,2} \tilde \beta_{j,l}v_2(\xi,x)\tilde v_l(\xi, y)^* \Big)
=O(|\xi|) ,
\ee
\be\label{newcrucial4}
\Big(\beta_{j,2} \tilde \beta_{j,l}v_2(\xi,x)\tilde v_l(\xi, y)^* \Big)
(0,1)^T
=O(|\xi|) 
\ee

From representation (\ref{Gdiff}), bounds \eqref{crucial}--\eqref{crucial2},
and $\Re \lambda_j(\xi)\le -\theta |\xi|^2$,
we obtain by the triangle inequality
\be
\|\tilde G^1(\cdot,t;\cdot)\|_{L^\infty(x,y)}=\|G^I-\bar U' \tilde e\|_{L^\infty(x,y)}\le C\|e^{-\theta |\xi|^2 t} \phi(\xi)\|_{L^1(\xi)}
 \le  C (1+t)^{-\frac{1}{2}}.
\ee
Derivative bounds follow similarly,
since $x$-derivatives falling on $v_{jk}$ are harmless, whereas,
by \eqref{crucial3}--\eqref{crucial4},
$y$- or $t$-derivatives falling on $\tilde v_{jl}$
or on $e^{i\xi\cdot(x-y)}$ bring down a factor
of $|\xi|$ improving the decay rate by factor $(1+t)^{-1/2}$.
(Note that $|\xi|$ is bounded because of the cutoff function $\phi$,
so there is no singularity at $t=0$.)

To obtain the corresponding bounds for $p=2$, we note that (\ref{Gnew})
may be viewed itself as a Bloch decomposition with respect to variable
$z:=x-y$, with $y$ appearing as a parameter.
Recalling (\ref{iso}), we may thus estimate
\ba
\sup_y &\|G^I(\cdot,t;y)-\bar U'\tilde e(\cdot, t;y)\|_{L^2(x)}
\le\\
&
C \sum_{j, k\ne 1, l}
\sup_y \|\phi(\xi) e^{\lambda_j(\xi)t}
v_k(\cdot, z_1)\tilde v_l^*(\cdot, y)
\tilde v_l(\cdot, y)^* \|_{L^2(\xi; L^2(z_1\in [0,X]))}\\
&\quad +
C\sum_{j, l} \sup_y \left\|\phi(\xi) e^{\lambda_j(\xi)t}
\Big( \frac{v_n(\cdot,x)-\bar U'(x)}{|\cdot|}\Big)
\tilde v_l(\cdot, y)^* \right\|_{L^2(\xi; L^2(z_1\in [0,X]))}\\
&\le
C \sum_{j, k\ne 1, l} \sup_y \|\phi(\xi) e^{-\theta |\xi|^2t} \|_{L^2(\xi)}
\sup_\xi\| v_k(\cdot, z_1) \|_{L^2(0,X)}
\| \tilde v_l(\cdot, y)^* \|_{L^\infty(0,X)}
\\
&\quad +
C\sum_{j, l} \sup_y \|\phi(\xi) e^{-\theta |\xi|^2t} \|_{L^2(\xi)}
\sup_\xi\left\| \Big( \frac{v_n(\xi,x)-\bar U'(x)}{|\xi|}\Big) \right\|_{L^2(0,X)}
\|\tilde v_l(\cdot, y)^* \|_{L^\infty(0,X)} \\
&\le
 C (1+t)^{-\frac{1}{4}},
\ea
where we have used in a crucial way the boundedness of $\tilde v_l$
in $L^\infty$,\footnote{This is clear for $\xi=0$, since $v_j$
are linear combinations of genuine and generalized eigenfunctions,
which are solutions of the homogeneous or inhomogeneous eigenvalue ODE.
More generally, note that the resolvent of $L_\xi-\gamma$
gains one derivative, hence the total eigenprojection, as a contour
integral of the resolvent, does too- now, use the one-dimensional
Sobolev inequality for periodic boundary conditions
to bound the $L^\infty$ difference from the
mean by the (bounded) $H^1$ norm, then bound the mean by the $L^1$ norm,
which is controlled by the $L^2$ norm.}
and also the boundedness of
$$
\Big( \frac{v_n(\xi,x)-\bar U'(x)}{\xi}\Big)
\sim
\partial_{\xi}v_n(r)
$$
in $L^2$, where $r\in (0,\xi)$.
Derivative bounds follow similarly as above, noting that
$y$- or $t$-derivatives bring down a factor $\xi$, while
$x$-derivatives are harmless, to obtain an additional factor
of $(1+t)^{-1/2}$ decay.
Finally, bounds for $2\le p\le \infty$ follow by $L^p$-interpolation.

Now, defining
\be\label{edef}
e(x,t;y):= \chi(t)\tilde e(x,t;y),
\ee
where $\tilde e$ is defined in \eqref{enew} and $\chi$ is a smooth cutoff function
such that $\chi(t)\equiv 1$ for $t\ge 2$ and $\chi(t)\equiv 0$ for $t\le 1$,
and setting $\tilde G:=G-\bar U'(x)e(x,t;y)$,
we readily obtain the estimates \eqref{GIest} by combining
the above estimates on $G^I-\bar U \tilde e$
with bound \eqref{SIIest} on $G^{II}$.

Finally, recalling, by Lemma \ref{blochfacts}, that $\tilde v_l\equiv \const$
for $l\ne 1$ while $\tilde \beta_{j,1}=O(|\xi|)$, we have
$$
\partial_{y} \Big( \beta_{j,1} \tilde \beta_{j,l}\tilde v_l(\xi, y)^*\Big)
=o(|\xi|).
$$
Bounds \eqref{ederest} thus
follow from \eqref{enew} by the argument
used to prove \eqref{GIest}, together with the observation that
$x$- or $t$-derivatives bring down factors of $\xi$.
Bounds \eqref{eest} follow similarly for $j+l\ge 1$,
in which case the integrand on the righthand side of \eqref{enew}
(now differentiated in $x$ and or $t$) is Lebesgue integrable.

In the critical case $j=l=0$,
taking $t$ without loss of generality $\ge 1$,
expanding
$$
\lambda_j(\xi)=-i\xi a_j - b_j \xi^2 +O(\xi^3),
$$
and setting $\check \lambda(\xi):= -i\xi a_j -b_j\xi^2$,
we may write $\tilde e(x,t;y)$ in \eqref{enew} as
\be\label{princterm}
\begin{aligned}
\Big(\frac{1}{2\pi }\Big)&\int_{\R}\sum_{j}
\check \beta_{j,1}(0) \tilde \beta_{j,2}(0)\tilde v_2(0, y)^*e^{i\xi \cdot (x-y)}
\xi^{-1}e^{\check\lambda_j(\xi)t}
d\xi\\
&=\Big(\frac{1}{2\pi }\Big)\pv\int_{\R}\sum_{j}
\check \beta_{j,1}(0) \tilde \beta_{j,2}(0)\tilde v_2(0, y)^*e^{i\xi \cdot (x-y)}
\xi^{-1}e^{\check\lambda_j(\xi)t}
d\xi\\
&=\sum_{j}
\check \beta_{j,1}(0) \tilde \beta_{j,2}(0)\tilde v_2(0, y)^*
\Big(\frac{1}{2\pi }\Big) \pv\int_{\R}
e^{i\xi \cdot (x-y)}
\xi^{-1}e^{\check\lambda_j(\xi)t}
d\xi,
\end{aligned}
\ee
where $\check \beta_{j,1}(0):=\lim_{\xi\to 0}(\xi \beta_{j,1}(\xi))$,
and the above series is convergent by the alternating series test,
plus a negligible error term
$$
\Big(\frac{1}{2\pi }\Big) \pv\int_{\R}
e^{i\xi \cdot (x-y)}
\phi(\xi)
O(e^{-\theta|\xi|^2t}) d\xi
$$
for which the integrand is Lebesgue integrable,
hence, by the previous argument, obeys the bounds for $j+l=1$.
(Note
that the integral on the lefthand side of \eqref{princterm}
is absolutely convergent by
$
\xi^{-1}(e^{-ia_1\xi t}-e^{-ia_2\xi t})\sim |a_1-a_2|t,
$
becoming conditionally convergent only when the integrand is split into
different eigenmodes.)

By (D2), we have $a_j$ real and $\Re b_j>0$.
Moreover, the operator
$L$, since real-valued, has spectrum with complex conjugate symmetry,
hence $b_j$ is real as well.
Observing that
$\Big(\frac{1}{2\pi }\Big) \pv\int_{\R}
e^{i\xi \cdot (x-y)}
\xi^{-1}e^{\check\lambda_j(\xi)t} d\xi$
is an antiderivative in $x$ of the inverse Fourier transform
$\Big(\frac{1}{2\pi }\Big)\int_{\R}
e^{i\xi \cdot (x-y)}
e^{\check\lambda_j(\xi)t} d\xi
=\frac{e^{-(x-y-a_jt)^2/4b_jt}}{\sqrt{4\pi b_jt}}$,
a Gaussian, we find that the principal part \eqref{princterm}
is a sum of errorfunctions
\be\label{eq:errfns}
\sum_{j=1}^2 c_j \, \errfn \Big( \frac{x-y-a_jt}{\sqrt{4b_jt}}, t \Big)
{\tilde v}_2(0,y),
\ee
hence bounded in $L^\infty$ as claimed, where $a_j$ denote the
characteristic speeds of the Whitham averaged system and
(on further inspection) $\sum_j c_j=0$.
This verifies bound \eqref{eest} in the final
case $j=l=0$, completing the proof.
\end{proof}

\br\label{ozrmk}
\textup{
See the proof of Proposition 1.5, \cite{OZ2}, for an essentially
equivalent estimate from the inverse Laplace transform
point of view
of the critical $\xi^{-1}$ contribution \eqref{princterm}.
}
\er

\subsection{Final linearized bounds}\label{s:finallin}

\begin{cor}\label{greenbds}
Under assumptions (H1)--(H4), (D1)--(\DDD),
the Green function $G(x,t;y)$ of \eqref{e:lin} decomposes as
$G=E+\tilde G$,
\be\label{E}
E=\bar U'(x)e(x,t;y),
\ee
where, for some $C>0$, all $t>0$, $1\le q\le 2\le p\le \infty$, $0\le j,k, l$,
$j+l\le K+1$, $1\le r\le 2$,
\ba\label{sheatbds}
\left\|\int_{-\infty}^{+\infty} \tilde G(x,t;y)f(y)dy\right\|_{L^p(x)}&\le
C (1+t)^{-\frac{1}{2}(1/q-1/p)}
\|f\|_{L^q\cap H^1},\\
\left\|\int_{-\infty}^{+\infty} \partial_y^r \tilde G(x,t;y)f(y)dy\right\|_{L^p(x)}&\le
C (1+t)^{-\frac{1}{2}(1/q-1/p)-\frac{1}{2}}
\|f\|_{L^q\cap H^{r+1}},\\
\left\|\int_{-\infty}^{+\infty} \partial_t^r \tilde G(x,t;y)f(y)dy\right\|_{L^p(x)}&\le
C (1+t)^{-\frac{1}{2}(1/q-1/p)-\frac{1}{2}} \|f\|_{L^q\cap H^{2r+1}},\\
\left\|\int_{-\infty}^{+\infty}  \tilde G(x,t;y)(0,1)^T f(y)dy\right\|_{L^p(x)}&\le
C (1+t)^{-\frac{1}{2}(1/q-1/p)-\frac{1}{2}} \|f\|_{L^q\cap H^1}.\\
\ea
\ba\label{etbds}
\left\|\int_{-\infty}^{+\infty} \partial_x^j\partial_t^k e(x,t;y)f(y)dy\right\|_{L^p}
&\le
(1+t)^{-\frac{1}{2}(1/q-1/p) -\frac{(j+k)}{2} +\frac{1}{2} }\|f\|_{L^q},\\
\left\|\int_{-\infty}^{+\infty} \partial_x^j\partial_t^k\partial_y^r e(x,t;y)f(y)dy\right\|_{L^p}
&\le
(1+t)^{-\frac{1}{2}(1/q-1/p) -\frac{(j+k)}{2} }\|f\|_{L^q},\\
\left\|\int_{-\infty}^{+\infty} \partial_x^j\partial_t^k e(x,t;y)
(0,1)^Tf(y)dy\right\|_{L^p}
&\le
(1+t)^{-\frac{1}{2}(1/q-1/p) -\frac{(j+k)}{2} }\|f\|_{L^q}.\\
\ea
Moreover, $e(x,t;y)\equiv 0$ for $t\le 1$.
\end{cor}

\begin{proof}
({\it Case $q=1$}).
From (\ref{GIest}) and the triangle inequality we obtain
$$
\Big\|\int_{\R}\tilde G^I(x,t;y)f(y)dy\Big\|_{L^p(x)}
\le
\int_{\R}\sup_y \|\tilde G^I(\cdot ,t;y)\|_{L^p}|f(y)|dy
\le C(1+t)^{-\frac{1}{2}(1-1/p)}\|f\|_{L^1}
$$
and similarly for $y$- and $t$-derivative estimates, and products
with $(0,1)^T$, which,
together with \eqref{SIIest}, yield \eqref{sheatbds}.
Bounds \eqref{etbds} follow similarly by the triangle inequality
and \eqref{ederest}--\eqref{eest}.

({\it Case $q=2$}).
From \eqref{crucial}--\eqref{crucial2}, and analyticity of
$v_j$, $\tilde v_j$, we have boundedness
from $L^2([0,X])\to L^2([0,X])$
of the projection-type operators
\be\label{crucialx}
f\to
\beta_{j,n} \tilde \beta_{j,l}
\Big(v_n(\xi,x)-\bar U'(x)\Big)
\langle \tilde v_l, f\rangle
\ee
and
\be\label{crucial2x}
f\to
\beta_{j,k} \tilde \beta_{j,l}v_k(\xi,x)
\langle \tilde v_l,f\rangle
 \; \hbox{\rm for }\; k\ne 1,
\ee
uniformly with respect to $\xi$,
from which we obtain by \eqref{Gdiff}, \eqref{edef}, and (\ref{iso})
the bound
\be\label{triv}
\left\|\int_{-\infty}^{+\infty} \tilde G^I(x,t;y)f(y)dy\right\|_{L^2(x)}\le
C\|f\|_{L^2(x)},
\ee
for all $t\ge 0$, yielding together with \eqref{SIIest}
the result \eqref{sheatbds} for $p=2$, $r=1$.
Similarly, by boundedness of $\tilde v_j$, $v_j$, $\bar U'$
in all $L^p[0,X]$, we have
$$
\begin{aligned}
\left\|e^{\lambda_j(\xi)t}
\beta_{j,n} \tilde \beta_{j,l}
\Big(v_n(\xi,x)-\bar U'(x)\Big)
\langle \tilde v_l, \hat f\rangle\right\|_{L^\infty(x)}
& \le Ce^{-\theta |\xi|^2t} \|\hat f(\xi,\cdot)\|_{L^2(x)},\\
\left\|e^{\lambda_j(\xi)t}
\beta_{j,k} \tilde \beta_{j,l}v_k(\xi,x)
\langle \tilde v_l, \hat f\rangle \right\|_{L^\infty(x)}
& \le Ce^{-\theta |\xi|^2t} \|\hat f(\xi,\cdot)\|_{L^2(x)},
 \; \hbox{\rm for }\; k\ne 1,
\end{aligned}
$$
$C,\, \theta>0$, yielding by definitions \eqref{Gdiff}, \eqref{edef} the bound
\ba\label{last}
\left\|\int_{-\infty}^{+\infty} \tilde G^I(x,t;y)f(y)dy\right\|_{L^\infty(x)}
&\le
\Big(\frac{1}{2\pi }\Big) \int_{-\pi}^{\pi}
C\phi(\xi) e^{-\theta |\xi|^2t}\|\hat f(\xi,\cdot)\|_{L^2(x)}
d\xi\\
&\le C\left\|\phi(\xi) e^{-\theta |\xi|^2t}\right\|_{L^2(\xi)} \|\hat f\|_{L^2(\xi, x)}\\
&\le C(1+t)^{-\frac{d}{4} } \|f\|_{L^2([0,X])},
\ea
hence giving the result for $p=\infty$, $r=0$.
The result for $r=0$ and general $2\le p\le \infty$ then follows by
$L^p$ interpolation between $p=2$ and $p=\infty$.
Derivative bounds $1\le r\le 2$ follow by
similar arguments, using \eqref{crucial3}--\eqref{crucial4},
as do bounds for products with $(0,1)^T$.
Bounds \eqref{etbds} follow similarly.

({\it Case $1\le q \le 2$}).
By Riesz--Thorin interpolation between the cases $q=1$ and $q=2$,
we obtain the bounds asserted in the general case $1\le q\le 2$,
$2\le p\le \infty$.
\end{proof}

Note the close analogy between the bounds of Corollary \ref{greenbds}
and those obtained in \cite{MaZ3,MaZ1} for the viscous or relaxation
shock wave case.

\section{Nonlinear stability}\label{s:nonlin}
With the bounds of Corollary \eqref{greenbds}, nonlinear
stability follows by a combination of the argument of \cite{JZ3,JZ4}
and modifications introduced in the shock wave case
to treat partial parabolicity and potential loss of derivatives
in the nonlinear iteration scheme \cite{Z1,Z3}.

\subsection{Nonlinear perturbation equations}\label{s:pert}

Given a solution $\tilde U(x,t)$ of \eqref{eqn:1conslaw},
define the nonlinear perturbation variable
\be\label{pertvar}
v=U-\bar U=
\tilde {U}(x+\psi(x,t),t)-\bar U(x),
\ee
where
\be\label{uvar}
U(x,t):=\tilde {U}(x+\psi(x,t),t)
\ee
and $\psi:\RM\times\RM\to\RM$ is to be chosen later.

\begin{lem}\label{4.1}
For $v$, $U$ as in \eqref{pertvar}, \eqref{uvar}, and $|\tilde U|$ bounded,
\begin{equation}\label{eqn:1nlper}
U_t+f(U)_{x}-(B(U)U_{x})_x-g(U)=\left(\partial_t-L\right)\bar{U}'(x)\psi(x,t)
+P
+\partial_x R +
\partial_t S ,
\ee
where
\begin{equation}\label{eqn:P}
P=\left(g(\tilde U)-g(\bar{U})\right)\psi_x=(0,1)^T\mathcal{O}(|v||\psi_x|),
\end{equation}
\begin{equation}\label{eqn:R}
R:= v\psi_t + B(\tilde{U})(\bar U_x +v_x)\frac{\psi_x^2}{1+\psi_x}-\left(B(\tilde U)-B(\bar U)\right)\bar U_x\psi_x-B(\tilde{U})v_x\psi_x,
\end{equation}
and
\begin{equation}\label{eqn:S}
S
:=- v\psi_x.
\end{equation}
\end{lem}

\begin{proof}
By the definition of $U$ in \eqref{uvar} we have by a
straightforward computation
\begin{align*}
U_t(x,t)&=\tilde{U}_x(x+\psi(x,t),t)\psi_t(x,t)+\tilde{U}_t(x+\psi,t)\\
f(U(x,t))_x&=df(\tilde{U}(x+\psi(x,t),t))\tilde{U}_x(x+\psi,t)\cdot(1+\psi_x(x,t))\\
U_{x}(x,t)&=\tilde{U}_x(x+\psi(x,t),t)\cdot(1+\psi_x(x,t)).
\end{align*}

By $\tilde U_t + df(\tilde{U})\tilde{U}_x-
(B(\tilde U) \tilde{U}_{x})_x-g(\tilde U)=0$,
it follows that
\ba\label{altform}
U_t+f(U)_{x}-(B(U)U_{x})_x-g(U)
&=\tilde{U}_x\psi_t+df(\tilde{U})\tilde{U}_x\psi_x
-(B(\tilde U)\tilde U_x)_x \psi_x
-(B(\tilde U)\tilde{U}_x\psi_x)_x\\
&= \tilde U_x \psi_t -\tilde U_{t} \psi_x
+g(\tilde U)\psi_x
-(B(\tilde U)\tilde{U}_x\psi_x)_x,
\ea
where it is understood that derivatives of $\tilde U$ appearing
on the righthand side
are evaluated at $(x+\psi(x,t),t)$.
Moreover, by another direct calculation,
using $L(\bar{U}'(x))=0$, we have
$$
\begin{aligned}
\left(\partial_t-L\right)\bar{U}'(x)\psi&=
\bar{U}_x\psi_t -\bar{U}_t\psi_{x}
+df(\bar {U})\bar {U}_x\psi_x -(B(\bar  U)\bar  U_x)_x \psi_x
-(B(\bar U)\bar{U}_x\psi_x)_x\\
&= \bar U_x \psi_t -\bar U_{t} \psi_x
+g(\bar U)\psi_x
-(B(\bar U)\bar{U}_x\psi_x)_x.
\end{aligned}
$$
Subtracting, and using the facts that,
by differentiation of $(\bar U+ v)(x,t)= \tilde U(x+\psi(x,t),t)$,
\ba\label{keyderivs}
\bar U_x + v_x&= \tilde U_x(1+\psi_x),
\qquad
\bar U_t + v_t= \tilde U_t + \tilde U_x\psi_t,\\
\ea
so that
\ba\label{solvedderivs}
\tilde U_x-\bar U_x -v_x&=
-(\bar U_x+v_x) \frac{\psi_x}{1+\psi_x},
\qquad
\tilde U_t-\bar U_t -v_t=
-(\bar U_x+v_x) \frac{\psi_t}{1+\psi_x},\\
\ea
we obtain
\ba
U_t+ f(U)_{x} - (B(U)U_{x})_x-g(U)&=
(\partial_t-L)\bar{U}'(x)\psi
+v_x\psi_t - v_t \psi_x
\\ &\quad +
\big(g(\tilde U)-g(\bar U)\big) \psi_x-\left(B(\tilde{U})v_x\psi_x\right)_x
\\ &\quad +
\left(B(\tilde U)(\bar U_x +v_x)\frac{\psi_x^2}{1+\psi_x} \right)_x
\\ &\quad
-\left(\left(B(\tilde U)-B(\bar U)\right)\bar U_x\psi_x\right)_x,
\ea
yielding \eqref{eqn:1nlper} by
$v_x\psi_t - v_t \psi_x = (v\psi_t)_x-(v\psi_x)_t$.
\end{proof}

\begin{cor}\label{4.2}
The nonlinear residual $v$ defined in \eqref{pertvar} satisfies
\be\label{veq}
v_t-Lv=\left(\partial_t-L\right)\bar{U}'(x)\psi
-Q_{x}+T+ P+ R_x +
\partial_t S,
\ee
where $P$, $R$, and $S$ are as in Lemma \ref{4.1} and $Q$ and $T$ are defined by
\ba\label{eqn:Q}
Q:&=f(\tilde{U}(x+\psi(x,t),t))-f(\bar{U}(x))-df(\bar{U}(x))v\\
&\quad\quad -\left(B(\tilde{U}(x+\psi(x,t),t))\tilde{U}_x(x+\psi(x,t),t)-B(\bar{U}(x))\bar{U}_x(x)\right)\\
&\quad\quad\quad\quad-\left(B(\bar{U})v_x+\left(dB(\bar{U})\bar{U}_x\right)v\right)\\
\ea
and
\be\label{eqn:T}
T:=g(\tilde{U}(x+\psi(x,t),t))-g(\bar{U}(x))-dg(\bar{U}(x))v=
(0,1)^T\mathcal{O}(|v|^2),
\ee
\end{cor}

\begin{proof}
Taylor expansion comparing \eqref{eqn:1nlper} and
$\bar U_t + f(\bar U)_x-(B(\bar U)\bar U_{x})_x-g(\bU)=0$.
\end{proof}

\subsection{Cancellation estimate}\label{s:cancellation}

Our strategy in writing \eqref{veq} is motivated by the following
basic cancellation principle.

\begin{prop}[\cite{HoZ}]\label{p:cancellation}
For any $f(y,s)\in L^p \cap C^2$ with $f(y,0)\equiv 0$, there holds
\be\label{e:cancel}
\int^t_0 \int G(x,t-s;y) (\partial_s - L_y)f(y,s) dy\,ds
= f(x,t).
\ee
\end{prop}

\begin{proof} Integrating the left hand side by parts, we obtain
\be
\int G(x,0;y)f(y,t)dy - \int G(x,t;y)f(y,0)dy
+ \int^t_0 \int
(\partial_t - L_y)^*G(x,t-s;y) f(y,s)dy\, ds.
\label{5.53.2}
\ee
Noting that, by duality,
$$
(\partial_t - L_y)^* G(x,t-s;y) = \delta(x-y) \delta(t-s),
$$
$\delta(\cdot)$ here denoting the Dirac delta-distribution,
we find that the third term on the righthand side
vanishes in \eqref{5.53.2}, while,
because $G(x,0;y) = \delta(x-y)$, the first term is simply $f(x,t)$.
The second term vanishes by $f(y,0)\equiv 0$.
\end{proof}


\subsection{Nonlinear damping estimate}

The following technical result is a key ingredient in
the nonlinear stability analysis that follows.
Applying Duhamel's principle to \eqref{veq} and using Proposition \ref{p:cancellation}
yields%
%
\ba\label{prelim}
  v(x,t)&=\int^\infty_{-\infty}G(x,t;y)v_0(y)\,dy  \\
  &\quad
  + \int^t_0 \int^\infty_{-\infty} G(x,t-s;y)
  (-Q_y+T+ R_y + S_s ) (y,s)\,dy\,ds
+ \psi (t) \bar U'(x).
\ea
Note that terms $Q_y$ and $S_s$ involve derivatives of $v$ (respectively
second derivative in space and first derivative in time) of maximal order,
hence to close a nonlinear iteration scheme based on \eqref{prelim}
would appear to require delicate maximal regularity estimates
rather than the straightforward ones that we have obtained.
Indeed, estimated using the linearized bounds of Corollary \ref{greenbds},
the righthand side
appears to lose several degrees of regularity as a function
from $H^K\to L^2$ of $v$.
However, the next proposition,
adapted from the methods
%
of \cite{MaZ4,Z3},
shows that higher-order derivatives are slaved to lower-order ones,
hence derivatives ``lost'' at the linearized level may be
``regained'' at the nonlinear level.
This effectively separates the issues of decay and regularity,
allowing us to close a nonlinear iteration without the use of
maximal regularity estimates or a more complicated quasilinear iteration scheme.

\begin{proposition}\label{damping}
Let $v_0\in H^K$ ($K$ as in (H1)), and suppose that
for $0\le t\le T$, the $H^K$ norm of $v$ and the $H^{K+1}$ norms of
$\psi_t(\cdot,t)$ and $\psi_x(\cdot,t)$
remain bounded by a sufficiently small constant.
Moreover, suppose that the Froude number $F$, viscosity $\nu$,
and velocity derivative $\bar u_x$  satisfy the amplitude condition
$\nu \bar u_x < F^{-1}$.
Then, there are constants
$C, \theta_{1}>0$ such that, for all $0\leq t\leq T$,
\begin{equation}\label{Ebounds}
\|v(t)\|_{H^K}^2 \leq C e^{-\theta_1 t} \|v(0)\|^2_{H^K} +
C \int_0^t e^{-\theta_1(t-s)} \left(\|v\|_{L^2}^2 +
\|(\psi_t, \psi_x)\|_{H^K}^2 \right) (s)\,ds.
\end{equation}
\end{proposition}

The proof of this result will be given in Appendix \ref{s:energy}.
Here,
we briefly outline the main ideas.  First, notice that by
subtracting from the equation \eqref{altform} for $U$
the equation for $\bar U$, we may write the
nonlinear perturbation equation as
\ba\label{vperturteq}
v_t + (Av)_x-(Bv_{x})_x-Cv&= P- Q(v)_x +T(v)
+ \tilde U_x \psi_t -\tilde U_{t} \psi_x \\
&\quad +g(\bar{U})\psi_x-\left(B(\tilde{U})\tilde{U}_x\psi_x\right)_x
\ea
where $A$, $B$, $C$ are as in \eqref{coeffs},
$P$, $Q$, and $T$ are as in
Corollary \ref{4.2},
$g$ and $B$ are as in \eqref{ab},
and it is understood that derivatives of $\tilde U$ appearing
on the righthand side
are evaluated at $(x+\psi(x,t),t)$.
Using \eqref{solvedderivs} to replace $\tilde U_x$ and
$\tilde U_t$ respectively by
$\bar U_x + v_x -(\bar U_x+v_x) \frac{\psi_x}{1+\psi_x}$
and
$\bar U_t + v_t -(\bar U_x+v_x) \frac{\psi_t}{1+\psi_x}$,
and moving the resulting $v_t\psi_x$ term to the lefthand side
of \eqref{vperturteq}, we obtain
\ba\label{vperturteq2}
(1+\psi_x) v_t&=(Bv_{x})_x-(A v)_x+ Cv+ P- Q(v)_x +T(v)\\
&\quad\quad+ \left(\bar U_x+v_x\right)\psi_t+g(\bar{U})\psi_x
-\left(B(\tilde{U})\left(\bar{U}_x+v_x\right)\frac{\psi_x}{1+\psi_x}\right)_x
\ea

Define now the Friedrichs symmetrizer
\be\label{Sigma}
\Sigma:=\bp 1 & 0\\ 0& \delta^{-2}\ep,
\ee
where $\delta^2:= -A_{12}= \bar \tau^{-3}(F^{-1}-\nu \bar u_x)$.  By
\eqref{froudebd}, $\Sigma$ is a symmetric positive definite symmetrizer for the hyperbolic
part of \eqref{vperturteq2} in the sense that
$
\Sigma A=\bp -c & -1 \\-1 &-c\delta^{-2}\ep
$
is a symmetric matrix,
where $A$ is as in \eqref{coeffs}.
Furthermore, to compensate for the lack of total parabolicity
of the governing equation, here indicated by the presence of a neutral eigenspace of
the matrix $\Sigma B$, we
introduce
the skew-symmetric Kawashima compensator
\be\label{K}
K:=\eta\bp 0 & -1\\1 & 0\ep, \quad 0<\eta\ll 1
\ee
and note that, in particular, for $\eta>0$ sufficiently small there exists a constant $\theta>0$
such that %
$\Re (KA+\Sigma B)\geq\theta$.

Now defining the functional
\[
\mathcal{E}[v]:=\left<v,\Sigma v\right>+\sum_{j=1}^K\left(\left<\partial_x^jv,K\partial_x^{j-1}v\right>
                      +\left<\partial_x^jv,\Sigma\partial_x^jv\right>\right)
\]
where here $\left<\cdot,\cdot\right>$ denotes the standard $L^2(\RM^n)$ inner product,
we find by a direct
%
but lengthy calculation using Sobolev embedding and interpolation
to absorb nonlinear and intermediate-derivative terms that
\be\label{Eineq}
\partial_t \calE(v) \leq -{\theta_1} \|v\|_{H^K}^2 +
C\left( \|v\|_{L^2}^2+ \|(\psi_t, \psi_x)\|_{H^K(x,t)}^2 \right)
\ee
for some positive constants $C, \theta_1 >0$, so long as $\|\tilde{U}\|_{H^K}$ remains bounded
and the quantities $\eta>0$, $\|v\|_{H^K}$ and $\|(\psi_t,\psi_x)\|_{H^K(x)}$ remain sufficiently small.
By Cauchy--Schwarz and the fact that $\Sigma$ is positive definite by \eqref{froudebd}, we have  $\mathcal{E}(v)\sim\|v\|_{H^K}^2$ for $\eta>0$
sufficiently small and hence \eqref{Eineq} implies
\[
\partial_t \calE(v) \leq -{\theta_1}\calE(v)+
C\left( \|v\|_{L^2}^2+ \|(\psi_t, \psi_x)\|_{H^K(x,t)}^2 \right)
\]
from which (\ref{Ebounds}) follows by Gronwall's inequality and, again, the equivalence of $\CalE(v)$
and $\|v\|_{H^K}^2$.

For more details and a complete proof of the key inequality \eqref{Eineq}, see Appendix \ref{s:energy}.

\br\label{froudermk}
\textup{
The condition \eqref{froudebd} gives effectively an
upper bound on the allowable amplitude of the wave,
for fixed Froude number and viscosity.
It is not clear that this has any connection with behavior.
Certainly it is needed for our argument structure, and
perhaps even for the validity of \eqref{Ebounds}, which is
itself convenient but clearly not necessary for stability.
The resolution of this issue would be very interesting from the
standpoint of applications, both in this and related contexts.
}
\er

\br\label{damprmk}
\textup{
The Lagrangian formulation appears essential here
in order to obtain quantitative bounds like \eqref{froudebd}
on the amplitude of the wave.
One can carry out damping estimates
for
sufficiently small-amplitude waves in Eulerian coordinates by the
argument of  \cite{MaZ2} in the shock wave case; however,
the large-amplitude argument of \cite{MaZ4}, depending on
global noncharacteristicity of the wave-- corresponding here
to nonvanishing of $u-s$, where $s$ is wave speed
in Eulerian coordinates-- together with bounded variation
of $\bar U_x$,
%
appears to fail irreparably
in the periodic case.
As we have shown here, the same argument succeeds in Lagrangian
coordinates provided that the linearized convection matrix $A$
is symmetrizable (the meaning of bound \eqref{froudebd}).
For similar observations regarding the advantages for energy
estimates of the special structure in Lagrangian coordinates, see \cite{TZ1}.
}
\er

\subsection{Integral representation/$\psi$-evolution scheme}

Recalling the Duhamel representation \eqref{prelim} of the perturbation $v$
along with the decomposition $G=\bar U'(x)e+\tilde{G}$ of Corollary \ref{greenbds},
we find that defining $\psi$ implicitly as 
\ba
  \psi (x,t)& =-\int^\infty_{-\infty}e(x,t;y) U_0(y)\,dy \\
&\quad
  -\int^t_0\int^{+\infty}_{-\infty} e(x,t-s;y)
  (P-Q_y+ T+R_y + S_s ) (y,s)\,dy\,ds ,
 \label{psi}
\ea
where $e$ is defined as in \eqref{edef}, 
results in the {\it integral representation}
\ba \label{u}
  v(x,t)&=\int^\infty_{-\infty} \tilde G(x,t;y)v_0(y)\,dy \\
&\quad
  +\int^t_0\int^\infty_{-\infty}\tilde G(x,t-s;y)
  (P-Q_y+ T+R_y + S_s  ) (y,s)\,dy\,ds ,
\ea
for the nonlinear perturbation $v$; see \cite{Z1,MaZ2} for further details.  Furthermore, differentiating (\ref{psi}) with respect to $t$,
and recalling that $e(x,s;y)\equiv 0$ for $s \le 1 $,
\ba \label{psidot}
   \partial_t^j\partial_x^k \psi (x,t)&=-\int^\infty_{-\infty}\partial_t^j\partial_x^k
e(x,t;y) U_0(y)\,dy \\
&\quad
  -\int^t_0\int^{+\infty}_{-\infty} \partial_t^j\partial_x^k
e(x,t-s;y)
  (P-Q_y+ T+R_y + S_s) (y,s)\,dy\,ds .
  \ea
Equations \eqref{u}, \eqref{psidot}
together form a complete system in the variables $(v,\partial_t^j \psi,
\partial_x^k\psi)$,
$0\le j, k\le K+1$,
from the solution of which we may afterward recover the
shift $\psi$ via \eqref{psi}.
From the original differential equation \eqref{veq}
together with \eqref{psidot},
we readily obtain short-time existence and continuity with
respect to $t$ of solutions
$(v,\psi_t, \psi_x)\in H^K$
by a standard contraction-mapping argument based on \eqref{Ebounds},
\eqref{psi}, and \eqref{etbds}.

\subsection{Nonlinear iteration}

Associated with the solution $(U, \psi_t, \psi_x)$ of integral system
\eqref{u}--\eqref{psidot}, define
\ba\label{szeta}
\zeta(t)&:=\sup_{0\le s\le t}
 \|(v, \psi_t,\psi_x)\|_{H^K}(s)(1+s)^{1/4} .
\ea

\bl\label{sclaim}
For all $t\ge 0$ for which $\zeta(t)$ is
finite and sufficiently small,
some $C>0$,
and $E_0:=\|U_0\|_{L^1\cap H^K}$ sufficiently small,
\be\label{eq:sclaim}
\zeta(t)\le C(E_0+\zeta(t)^2).
\ee
\el

\begin{proof}
By \eqref{eqn:P}--\eqref{eqn:S} and \eqref{eqn:Q}--\eqref{eqn:T} 
and corresponding bounds on the derivatives together
with definition \eqref{szeta},
\ba\label{sNbds}
\|(P,Q,R,S,T)\|_{L^1\cap H^2}
&\le \|(v,v_x,\psi_t,\psi_x)\|_{L^2}^2+
\|(v,v_x,\psi_t,\psi_x)\|_{H^2}^2
\le C\zeta(t)^2 (1+t)^{-\frac{1}{2}},\\
\ea
so long as $|\psi_x|\le |\psi_x|_{H^K}\le \zeta(t)$ remains small.
%
Applying Corollary \ref{greenbds} with $q=1$ 
to representations
\eqref{u}--\eqref{psidot}, we obtain for any $2\le p<\infty$
\ba\label{sest}
\|v(\cdot,t)\|_{L^p(x)}& \le
C(1+t)^{-\frac{1}{2}(1-1/p)}E_0 \\
&\quad +
C\zeta(t)^2\int_0^{t} (1+t-s)^{-\frac{1}{2}(1/2-1/p)}(t-s)^{-\frac{3}{4}}
(1+s)^{-\frac{1}{2}}ds\\
&
\le
 C(E_0+\zeta(t)^2) (1+t)^{-\frac{1}{2}(1-1/p)}
\ea
and
\ba\label{sestad}
\|(\psi_t,\psi_x)(\cdot, t)\|_{W^{K+1,p}}& \le
C(1+t)^{-\frac{1}{2}}E_0 +
C\zeta(t)^2\int_0^{t} (1+t-s)^{-\frac{1}{2}(1-1/p)-1/2}
(1+s)^{-\frac{1}{2}}ds \\
&\le
 C(E_0+\zeta(t)^2) (1+t)^{-\frac{1}{2}(1-1/p)},
\ea
yielding in particular that $\|(\psi_t,\psi_x)\|_{H^{K+1}}$
be arbitrarily small, verifying the hypotheses of Proposition
\ref{damping}.\footnote{Note that we have gained a necessary one degree
of regularity in $\psi$, the regularity of $\psi$ being limited
only by the regularity of the coefficients of the underlying
PDE \eqref{eqn:1conslaw}.}
Using \eqref{Ebounds} and \eqref{sest}--\eqref{sestad},
we thus obtain
$\|v(\cdot,t)\|_{H^K(x)} \le
 C(E_0+\zeta(t)^2) (1+t)^{-\frac{1}{4}}$.
Combining this with \eqref{sestad}, $p=2$, rearranging, and recalling
definition \eqref{szeta}, we obtain \eqref{sclaim}.
\end{proof}

\begin{proof}[Proof of Theorem \ref{main}]
By short-time $H^K$ existence theory,
$\|(v,\psi_t,\psi_x)\|_{H^{K}}$ is continuous so long as it
remains small, hence $\zeta$ remains
continuous so long as it remains small.
By \eqref{sclaim}, therefore,
it follows by continuous induction that
$\zeta(t) \le 2C E_0$ for $t \ge0$, if $E_0 < 1/ 4C$,
yielding by (\ref{szeta}) the result (\ref{eq:smallsest}) for $p=2$.
Applying \eqref{sest}--\eqref{sestad}, we obtain
(\ref{eq:smallsest}) for $2\le p\le p_*$ for any $p_*<\infty$,
with uniform constant $C$.
Taking $p_*>4$ and estimating
\[
\|P\|_{L^2}, \, \|Q\|_{L^2}, \, \|R\|_{L^2}, \, \|S\|_{L^2},\, \|T\|_{L^2}(t)
\le \|(v,\psi_t,\psi_x)\|_{L^4}^2\le CE_0(1+t)^{-\frac{3}{4}}
\]
in place of the weaker \eqref{sNbds},
then applying Corollary \ref{greenbds} with $q=2$ 
we obtain finally \eqref{eq:smallsest} for $2\le p\le \infty$,
by a computation similar \eqref{sest}--\eqref{sestad};
we omit the details of this final bootstrap argument.
Estimate \eqref{eq:stab} then follows using \eqref{etbds} with
$q=1$, by
\ba\label{sesta}
\|\psi(t)\|_{L^p}& \le
C E_0 (1+t)^{\frac{1}{2p}}
+
C\zeta(t)^2\int_0^{t} (1+t-s)^{-\frac{1}{2}(1-1/p)}
(1+s)^{-\frac{1}{2}}ds\\
& \le
C(1+t)^{\frac{1}{2p}}(E_0+\zeta(t)^2),
\ea
together with the fact that
$ \tilde U(x,t)-\bar U(x)= v(x-\psi,t)+ \bar U(x)-\bar U(x-\psi), $
so that $|\tilde U(\cdot, t)-\bar U|$ is controlled
by the sum of $|v|$ and
$|\bar U(x)-\bar U(x-\psi)|\sim |\psi|$.
This yields stability for $|U-\bar U|_{L^1\cap H^K}|_{t=0}$
sufficiently small, as described in the final line of the theorem.
\end{proof}
\appendix

\section{Nonlinear energy estimate}\label{s:energy}

The goal of this
appendix is to prove the inequality \eqref{Eineq}, which was the key ingredient
in the nonlinear energy estimate in Proposition \eqref{damping}.  To this end, we write the nonlinear perturbation equation
\eqref{vperturteq2} for the variable $v=(\tau,u)^T$ as
\ba\label{linperturteq1}
(1+\psi_x) v_t=(Bv_{x})_x-(A v)_x+ Cv+ \left(\bar U_x+v_x\right)\psi_t+g(\bar{U})\psi_x+\mathcal{N}
\ea
where the function $\mathcal{N}=\mathcal{N}(v,\bar{U}_x,\psi_x,\psi_t)$ is defined by
\begin{equation}\label{N}
\mathcal{N}:=P- Q(v)_x +T(v)-\left(B(\tilde{U})\left(\bar{U}_x+v_x\right)\frac{\psi_x}{1+\psi_x}\right)_x
\end{equation}
where $P$, $Q$, and $T$ are defined as in \eqref{eqn:P}, \eqref{eqn:Q}, and \eqref{eqn:T}, respectively.
The key to the analysis is to carefully keep track of the ``hyperbolic" ($\tau$) and ``parabolic" ($u$) components of $v$ separately.
We begin by symmetrizing the linearized convection matrix $A$ of \eqref{linperturteq1} by introducing the Friedrichs
symmetrizer $\Sigma$ defined in \eqref{Sigma} as
\[
\Sigma=\bp 1&0\\ 0&\delta^{-2}\ep,\quad\delta^2=\bar{\tau}^{-3}\left(\frac{1}{F}-\nu\bar{u}_x\right),
 \]
noting in particular that it is symmetric positive definite by the amplitude condition \eqref{froudebd}.
The fact that
$
\Sigma A=\bp -c & -1 \\-1 &-c\delta^{-2}\ep
$
is symmetric then yields hyperbolic properties of the solution using straightforward energy estimates, integration by
parts, and  the Friedrichs symmetrizer relation
\[
\Re\left<U,SU_{x}\right>=-\frac{1}{2}\left<U,S_x U\right>
\]
valid for all self-adjoint operators $S\in \CM^{n\times n}$ and $U\in \CM^n$.
Furthermore, for convenience we provide here a list of the block-structure of the various matrices arising in the forthcoming proofs:
notice by definition that
\be\label{dictionary}
\quad B,\;B_x=\bp 0 &0\\ 0&*\ep,
\quad  A_x,\, A_{xx},= \bp 0 & 0\\ * & 0\ep,
\quad C,\;C_x=\bp 0&0\\*&*\ep,
\ee
which immediately yields
\begin{align}
\Sigma B,\;\Sigma B_x,\;\Sigma_x B,\;\Sigma_{xx}B,\;\Sigma_x B_x&=\bp 0&0\\ 0&*\ep\label{dictB}\\
\Sigma A_x,\, \Sigma_x A,\, \Sigma_{xx} A,\,
\Sigma_x A_x,\, \Sigma A_{xx} &= \bp 0 & 0\\ 0 & *\ep\label{dictA}
\end{align}
and
\be\label{dictC}
\Sigma C=\Sigma C_x=\bp 0&0\\ * &*\ep.
\ee
We will continually refer back to these dictionaries throughout the proofs in this appendix.

\br\label{genstructure}
\textup{
The apparently special structure leading to \eqref{dictionary}
is in fact a special case of the more general structure pointed
out in \cite{TZ1},\footnote{
See conditions (A1)--(A2) of \cite{TZ1,Z4}.
} shared by the equations of 1-D
gas dynamics, MHD, and viscoelasticity \cite{TZ1,BLeZ} when expressed in
Lagrangian coordinates.
}
\er

Defining the first-order ``Friedrichs bilinear form" as
\[
\mathcal{F}_1[v_1,v_2]:=\left<v_1,\Sigma v_2\right>+\left<\partial_x v_1,\Sigma \partial_xv_2\right>
\]
our first step in proving 
Proposition \eqref{damping} is to establish the following lemma.

\begin{lemma}\label{l:Fried1}
Let $v(\cdot,0)\in H^1$ and suppose that for $0\leq t\leq T$, the $H^1$ norm of $v$ and the $H^2$ norms of $\psi_x$ and $\psi_t$
remain bounded by a sufficiently small constant.  Moreover, suppose that the amplitude condition \eqref{froudebd} holds.
Then we have the first-order ``Friedrichs-type" estimate
\begin{equation}
\begin{aligned}\label{Fried1}
\frac{1}{2}\frac{d}{dt}\mathcal{F}_1[v,v]
         &\leq-\left<v_x,w\Sigma Bv_x\right>-\left<v_{xx},w\Sigma Bv_{xx}\right>\\
&\quad\quad+C_1\left(\|v\|_{L^2}^2+\frac{1}{\eps}\|u_x\|_{L^2}^2+\eps\|\tau_x\|_{L^2}^2\right)\\
&\quad\quad\quad\quad +C_2\|\psi_x\|_{H^2}\left(\|v\|_{H^1}^2+\|u_{xx}\|_{L^2}^2\right)\\
&\quad\quad\quad\quad\quad\quad+\frac{C_3}{\eps}\left(\|\psi_t\|_{H^1}^2+\|\psi_x\|_{H^1}^2\right)+\mathcal{F}_1[v,w\mathcal{N}]
\end{aligned}
\end{equation}
valid for all $0\leq t\leq T$, for some constants $C_{1,2}>0$ where $w:=(1+\psi_x)^{-1}\in L^\infty$.
\end{lemma}

\begin{proof}
First, notice from \eqref{linperturteq1} and the symmetry of $\Sigma$ we have
\begin{align*}
\frac{1}{2}\frac{d}{dt}\left<v,\Sigma v\right>&=\left<v,\Sigma v_t\right>\\
&=\left<v,w\Sigma\left((Bv_x)_x-(Av_x)+Cv+\bar{U}_x\psi_t+v_x\psi_t+g(\bar{U})\psi_x+\mathcal{N}\right)\right>
\end{align*}
where $w:=(1+\psi_x)^{-1}$.  Since $\Sigma B$ is symmetric by \eqref{dictB} then, we have
\begin{align*}
\left<v,w\Sigma(Bv_x)_x\right>&=-\left<(w\Sigma_x)v,Bv_x\right>-\left<v_x,w\Sigma Bv_x\right>
=\frac{1}{2}\left<v,\left(\left(w\Sigma\right)_xB\right)_xv\right>-\left<v_x,w\Sigma Bv_x\right>
\end{align*}
and similarly
\[
\left<v,w\Sigma\left(Av\right)_x\right>=\left<v,w\Sigma A_x v\right>-\frac{1}{2}\left<v,\left(w\Sigma A\right)_xv\right>.
\]
Furthermore, assuming that $\|\psi_t\|_{H^2}$ remains bounded we clearly have the estimate
\begin{align*}
\left<v,w\Sigma \left((\bar{U}_x+v_x)\psi_t+g(\bar{U})\psi_x\right)\right>&=\left<v,w\Sigma\bar{U}_x\psi_t\right>-\frac{1}{2}\left<v,\left(w\Sigma\psi_t\right)_xv\right>
                +\left<v,w\Sigma g(\bar{U})\psi_x\right>\\
&\lesssim \left(\|v\|_{L^2}^2+\|\psi_t\|_{L^2}^2+\|\psi_x\|_{L^2}^2\right)
\end{align*}
which, by using Cauchy-Schwarz, immediately yields the zeroth-order estimate
\[
\frac{1}{2}\frac{d}{dt}\left<v,\Sigma v\right>\leq-\left<v_x,w\Sigma B v_x\right>+C\left(\|v\|_{L^2}^2+\|\psi_t\|_{L^2}^2+\|\psi_x\|_{L^2}^2\right)+\left<v,w\Sigma\mathcal{N}\right>
\]
for some positive constant $C>0$.

Continuing, we find that
\begin{align*}
\frac{1}{2}\frac{d}{dt}\left<v_x,w\Sigma v_x\right>&=\left<v_x,w_x\Sigma\left((Bv_x)_x-(Av)_x+Cv+\left(\bar{U}_x+v_x\right)\psi_t+g(\bar{U})\psi_x\right)\right>\\
&\quad\quad+\left<v_x,w\Sigma\left((Bv_x)_{xx}-(Av)_{xx}+(Cv)_x+\left((\bar{U}_x+v_x)\psi_t+g(\bar{U})\psi_x\right)_{x}\right)\right>\\
&\quad\quad\quad\quad+\left<v_x,\Sigma\left(w\mathcal{N}\right)_x\right>\\
&=:I_1+I_2+\left<v_x,\Sigma\left(w\mathcal{N}\right)_x\right>.
\end{align*}
To estimate $I_1$, 
notice that \eqref{dictB} immediately yields
\begin{align*}
\left<v_x,w_x\Sigma(Bv_x)_x\right>&=\left<v_x,w_x\Sigma B_xv_x\right>+\left<v_{x},w_x\Sigma Bv_{xx}\right>
\lesssim\|w_x\|_{L^\infty}\|u_x\|_{H^1}^2. 
\end{align*}
and that, similarly, we have the estimates
\begin{align*}
\left<v_x,w_x\Sigma(Av)_x\right>,~\left<v_x,w_x\Sigma Cv\right>&\lesssim\|w_x\|_{L^\infty}\|v\|_{H^1}^2,
\end{align*}
by \eqref{dictA} and \eqref{dictC}.  Finally, noting that for $\|\psi_t\|_{L^\infty}$ bounded we have
\[
\left<v_x,w_x\Sigma\left((\bar{U}_x+v_x)\psi_t+g(\bar{U})\psi_x\right)\right>\lesssim\|w_x\|_{L^\infty}\left(\|v\|_{H^1}^2+\|\psi_t\|_{L^2}^2+\|\psi_x\|_{L^2}^2\right),
\]
we see that
together these yield the estimate
\[
I_1\lesssim\|w_x\|_{L^\infty}\left(\|v\|_{H^1}^2+\|u_{xx}\|_{L^2}^2+\|\psi_t\|_{L^2}^2+\|\psi_x\|_{L^2}^2\right).
\]

To obtain the analogous estimate on $I_2$, 
first notice that \eqref{dictB} and the boundedness of $\|w\|_{L^\infty}$, together with Young's inequality,
imply
%
\begin{align*}
\left<v_x,w\Sigma\left(Bv_x\right)_{xx}\right>&=-\left<(w\Sigma)_xv_x+w\Sigma v_{xx},(Bv_x)_x\right>\\
&=-\left<v_{xx},w\Sigma Bv_{xx}\right>-\left<v_{xx},w\Sigma B_xv_x\right>\\
&\quad\quad -\left<v_x,\left(w\Sigma\right)_{x}B_xv_x\right>-\left<v_x,\left(w\Sigma\right)_xBv_{xx}\right>\\
&\leq -\left<v_{xx},w\Sigma Bv_{xx}\right>+\widetilde{C}_1\left(\frac{1}{\eps}\|u_x\|_{L^2}^2+\eps\|u_{xx}\|_{L^2}^2\right)\\
&\quad\quad+\widetilde{C}_2\|w_x\|_{L^\infty}\left(\|u_x\|_{L^2}^2+\|u_{xx}\|_{L^2}^2\right)
\end{align*}
for some constants $\widetilde{C}_1,\widetilde{C}_2>0$, where $\eps>0$ is a sufficiently small constant to be chosen later.
Similarly, 
using \eqref{dictA} and \eqref{dictC} we find that
\begin{align*}
\left<v_x,w\Sigma\left(Av\right)_{xx}\right>&=\left<v_x,w\Sigma\left(A_{xx}v+A_xv_x\right)\right>-\frac{1}{2}\left<v_x,\left(w\Sigma A\right)_xv_x\right>\\
&\quad\quad\lesssim\|u\|_{H^1}^2+\|w_x\|_{L^\infty}\|v\|_{H^1}^2,\\
\left<v_x,w\Sigma \left(Cv\right)_x\right>&\lesssim\left(\|v\|_{L^2}^2+\frac{1}{\eps}\|u_x\|_{L^2}^2+\eps\|\tau_x\|_{L^2}^2\right).
\end{align*}
Finally, noting again that $\|\psi_t\|_{L^\infty}$ is bounded we find
that
\begin{align*}
\left<v_x,w\Sigma\left((\bar{U}_x+v_x)\psi_t\right)_x\right>&\lesssim
               \left((\eps+\|\psi_{xt}\|_{L^\infty})\|v_x\|_{L^2}^2+\frac{1}{\eps}\|\psi_t\|_{H^1}^2\right)
-\frac{1}{2}\left<v_x,\left(w\Sigma \psi_t\right)_xv_{x}\right>\\
&\lesssim\left(\eps+\|w_x\|_{L^\infty}+\|\psi_{xt}\|_{L^\infty}\right)\|v_x\|_{L^2}^2+\|u_x\|_{L^2}^2+\frac{1}{\eps}\|\psi_t\|_{H^1}^2
\end{align*}
and
\[
\left<v_x,w\Sigma\left(g(\bar{U})\psi_x\right)_x\right>\lesssim \eps\|v_x\|_{L^2}^2+\frac{1}{\eps}\|\psi_x\|_{H^1}^2.
\]
Therefore, by choosing $\|\psi_{t}\|_{H^2}<\eps$, so that $\|\psi_{xt}\|_{L^\infty}<\eps$ by Sobolev embedding, and noting
that $\|\psi_{xt}\|_{L^\infty}$ is bounded, we have
\begin{align*}
I_2&\leq-\left<v_{xx},w\Sigma Bv_{xx}\right>+\widetilde{C}\left(\|v\|_{L^2}^2+\frac{1}{\eps}\|u_x\|_{L^2}^2+
\eps\|\tau_x\|_{L^2}^2+\frac{1}{\eps}\|\psi_t\|_{H^1}^2+\frac{1}{\eps}\|\psi_x\|_{H^1}^2\right)\\
&\quad\quad+\|w_x\|_{L^\infty}\left(\|v\|_{H^1}^2+\|u_{xx}\|_{L^2}^2\right)
\end{align*}
for some constant $\widetilde{C}>0$, from which the lemma follows by noting that 
$\|w_x\|_{L^\infty}\lesssim\|\psi_x\|_{H^2}$ 
by Sobolev embedding.
\end{proof}

From Lemma \ref{l:Fried1} it follows that if the diffusion $\Sigma B$ were positive definite we would immediately have the bound
\[
\frac{1}{2}\frac{d}{dt}\mathcal{F}_1[v,v]
         \leq-\theta\|v\|_{H^1}^2+C_1\left(\|v\|_{L^2}^2+\|\psi_t\|_{H^1}^2+\|\psi_x\|_{H^1}^2\right)
+\mathcal{F}_1[v,w\mathcal{N}]
\]
by using Sobolev embedding and choosing $\eps>0,\|\psi_x\|_{H^2}$ sufficiently small, which, up to the
contribution of the nonlinear residual terms $\mathcal{N}$, has the form of the inequality
stated in Proposition \eqref{damping}; see the proof of Lemma \ref{l:kaw} below for
details on how this calculation would proceed.  However, the lack of total parabolicity in
the governing equation \eqref{eqn:1conslaw} is manifested here in the fact that the matrix
$\Sigma B$ is not positive definite, rather being only positive semi-definite with rank one.
In order to compensate for this ``degenerate diffusion", 
we introduce the Kawashima compensator $K$ defined in \eqref{K} as
\[
K:=\eta\left(
         \begin{array}{cc}
           0 & -1 \\
           1 & 0 \\
         \end{array}
       \right)
\]
where $0<\eta\ll 1$ is a small parameter which will be determined later.  The fact that the hyperbolic effects in \eqref{linperturteq1}
can compensate for this degeneracy in the diffusive term $\Sigma B$ is
the
point of the following lemma.

\begin{lemma}\label{l:kaw}
Assume the amplitude condition \eqref{froudebd} holds.  Then for $\eta>0$ sufficiently small,
the matrix $\Sigma B+KA$ is positive definite and, furthermore, the associated bilinear form
satisfies the coercivity estimate
\[
\left<\xi,\left(\Sigma B+KA\right)\xi\right>\geq\theta\left(\|\xi_2\|_{L^2}^2+\eta\|\xi_1\|_{L^2}^2\right)
\]
for some constant $\theta>0$ and all $\xi=\left(\xi_1,\xi_2\right)^T\in L^2(\RM)$.
\end{lemma}

The proof of Lemma \ref{l:kaw} is based on simple matrix perturbation argument and is omitted.  Defining now the first-order ``Kawashima bilinear form" as
\[
\mathcal{E}_1[v_1,v_2]:=\mathcal{F}_1[v_1,v_2]+\left<\partial_xv_{1},Kv_2\right>
\]
and noting the special structures
\be\label{dictK}
KB,\;KB_x=\eta\bp 0&*\\0&0\ep,\quad
KA_x=\eta\bp* &0\\0&0\ep,\quad
KC=\eta\bp * & *\\0&0\ep,
\ee
we have the following refinement of the first-order Friedrichs-type estimate in Lemma \ref{l:Fried1}.

\begin{lemma}
Under the same hypothesis of Lemma \eqref{l:Fried1} and for $\eta>0$ sufficiently small, we have the first-order ``Kawashima-type" estimate
\begin{align*}
\frac{1}{2}\frac{d}{dt}\mathcal{E}_1[v,v]&\leq -\theta_1\left(\|u_x\|_{H^1}^2+\eta\|\tau_x\|_{L^2}^2\right)+\frac{C}{\eta^2}\left(\|v\|_{L^2}^2+\|\psi_t\|_{H^1}^2+\|\psi_x\|_{H^1}^2\right)
                   +\mathcal{E}_1[v,w\mathcal{N}]
\end{align*}
for some  constants $\theta_1,C>0$.
\end{lemma}

\begin{proof}
Using \eqref{dictK} along with arguments similar to those as in Lemma \ref{l:Fried1}, we obtain the estimate
\begin{align*}
\frac{1}{2}\partial_t\left<v_x,Kv\right>&=\left<v_x,wK\left((Bv_x)_x-(Av)_x+Cv+\bar{U}_x\psi_t+v_x\psi_t+g(\bar{U})\psi_x+\mathcal{N}\right)\right>\\
&\leq C\eta\left(\frac{1}{\delta}\|v\|_{L^2}^2+\delta\|\tau_x\|_{L^2}^2+\frac{1}{\delta}\|u_{x}\|_{L^2}^2
       +\frac{1}{\delta}\|u_{xx}\|_{L^2}^2+\frac{1}{\delta}\left(\|\psi_t\|_{L^2}^2+\|\psi_x\|_{L^2}^2\right)\right)\\
&\quad\quad-\left<v_x,wKAv_x\right>+\left<v_x,Kw\mathcal{N}\right>
\end{align*}
for some positive constants $\theta_1,C>0$, and for any $\delta>0$ sufficiently small, where we have used
that $\|\psi_t\|_{L^\infty}\lesssim\|\psi_t\|_{H^1}$ can be chosen sufficiently small, say of order $\mathcal{O}(\delta)$.
Since $w\in L^\infty$ then, it follows from Lemma \ref{l:Fried1} that
\begin{align*}
\frac{1}{2}\partial_t\mathcal{E}_1[v,v]
            &\leq-\left<v_x,w\left(\Sigma B+KA\right)v_x\right>-\left<v_{xx}w\Sigma Bv_{xx}\right>\\
&\quad\quad+C_1\left(\left(\frac{\eta}{\delta}+1\right)\|v\|_{L^2}^2
                  +\left(\frac{\eta}{\delta}+\frac{1}{\eps}\right)\|u_x\|_{L^2}^2+\left(\eta\delta+\eps\right)\|\tau_x\|_{L^2}^2+\frac{\eta}{\delta}\|u_{xx}\|_{L^2}^2\right)\\
&\quad\quad\quad\quad +C_2\|\psi_x\|_{H^1}\left(\|v\|_{H^1}^2+\|u_{xx}\|_{L^2}^2\right)\\
&\quad\quad\quad\quad\quad\quad+C_3\left(\frac{1}{\eps}+\frac{1}{\delta}\right)\left(\|\psi_t\|_{H^1}^2+\|\psi_x\|_{H^1}^2\right)+\mathcal{E}_1[v,w\mathcal{N}]
\end{align*}
By Lemma \ref{l:kaw} then, we find that for $\eta>0$ sufficiently small, say $0<\eta<\eta_0$, we have the estimate
\[
-\left<v_x,w\left(\Sigma B+KA\right)v_x\right>-\left<v_{xx}w\Sigma Bv_{xx}\right>\leq-\theta\left(\|u_x\|_{L^2}^2+\eta\|\tau_x\|_{L^2}^2+\|u_{xx}\|_{L^2}^2\right)
\]
for some constant $\theta>0$.  Thus, by fixing $\delta$ and choosing $\eps=\eps(\eta)$ such that
\begin{equation}\label{eqn:bds1}
0<\delta=\frac{\theta}{2C_1}\quad\textrm{ and }0<\eps(\eta)=\frac{\eta\theta}{4}.
\end{equation}
we find
that
\[
\left(\eta(-\theta+C_1\delta)+\eps\right)\|\tau_x\|_{L^2}^2=-\frac{\theta\eta}{4}\|\tau_x\|_{L^2}^2.
\]
By subsequently requiring that the free parameter $\eta>0$ satisfy
\begin{equation}\label{eqn:bds2}
0<\eta\leq\min\left\{\frac{\theta\delta}{2C_1},\eta_0\right\}
\end{equation}
we similarly find that
\[
\left(-\theta+\frac{\eta}{\delta}C_1\right)\|u_{xx}\|_{L^2}^2\leq-\frac{\theta}{2}\|u_{xx}\|_{L^2}^2.
\]
from which it follows by the above requirements on the parameters $\eta$, $\eps(\eta)$, and $\delta$ that
\begin{align*}
\frac{1}{2}\frac{d}{dt}\mathcal{E}_1[v,v]&\leq-\hat\theta\left(\|u_x\|_{L^2}^2+\eta\|\tau_x\|_{L^2}^2+\|u_{xx}\|_{L^2}^2\right)\\
&\quad\quad+\widetilde{C}_1\left(\|v\|_{L^2}^2+\frac{1}{\eta}\|u_{x}\|_{L^2}^2\right)
+C_2\|\psi_x\|_{H^1}\left(\|v\|_{H^1}^2+\|u_{xx}\|_{L^2}^2\right)\\
&\quad\quad\quad\quad\quad\quad+\frac{\widetilde{C}_3}{\eta}\left(\|\psi_t\|_{H^1}^2+\|\psi_x\|_{H^1}^2\right)+\mathcal{E}_1[v,w\mathcal{N}].
\end{align*}
Next, the Sobolev inequality $\|g_x\|_{L^2}^2\leq \|g_{xx}\|_{L^2}\|g\|_{L^2}$ along with Young's inequality implies that
$
\frac{\widetilde{C}_1}{\eta}\|u_x\|_{L^2}^2\leq\frac{\theta_1}{2}\|u_{xx}\|_{L^2}^2+\frac{2\widetilde{C_1}}{\hat\theta\eta^2}\|u\|_{L^2}^2
$
which, by now choosing $\|\psi_x\|_{H^1}^2$ sufficiently small so that
$
-\frac{\hat\theta}{2}+C_2\|\psi_x\|_{H^1}^2<0,
$
completes the proof.
\end{proof}

Using similar arguments, we can obtain higher order Kawashima--type estimates 
by defining the $k^{\textrm{th}}$-order
Kawashima bilinear form as
\[
\mathcal{E}_k[v_1,v_2]:=\left<v,\Sigma v\right>+\sum_{j=1}^k\left(\left<\partial_x^jv_1,K\partial_x^{j-1}v_2\right>
                      +\left<\partial_x^jv_1,\Sigma\partial_x^jv_2\right>\right)
\]
for each $k\in\mathbb{N}$.  Indeed, the following estimate can be obtained by simply iterating the above
argument and using the Sobolev inequality
$\|g_x\|_{H^j}\leq \alpha\|\partial_x^{j+2}g\|_{L^2}+\alpha^{-1}\|g\|_{L^2}$ for $\alpha>0$ sufficiently small.

\begin{lemma}\label{l:gen1}
Let $j\in\mathbb{N}$ and $v(\cdot,0)\in H^j$, and suppose that for $0\leq t\leq T$, the $H^j$ norm of $v$ and the $H^{j+1}$ norms
of $\psi_x$ and $\psi_t$ remain bounded by a sufficiently small constant.  Moreover, suppose that condition \eqref{froudebd} is
satisfied.  Then for $\eta>0$ sufficiently small, there exist constants $\theta_1, C>0$ such that, for all $0\leq t\leq T$,
\[
\frac{1}{2}\frac{d}{dt}\mathcal{E}_j[v,v]\leq-\theta_1\left(\|u_x\|_{H^{j}}^2+\eta\|\tau_x\|_{H^{j-1}}^2\right)
                +\frac{C}{\eta^2}\left(\|v\|_{L^2}^2+\|\psi_t\|_{H^j}^2+\|\psi_x\|_{H^j}^2\right)+\mathcal{E}_j[v,w\mathcal{N}].
\]
\end{lemma}


To complete the proof of Proposition \ref{damping} it remains to estimate the terms $\mathcal{E}_j[v,w\mathcal{N}]$ corresponding to the
nonlinear residual terms in the perturbation equation \eqref{linperturteq1}.  In particular, our goal is to
demonstrate that these terms can be absorbed into the bound already computed, in the sense that there
exists constants $C>0$ and $0<\eps\ll 1$ such that
\begin{equation}\label{eqn:absorb}
\mathcal{E}_j[v,w\mathcal{N}]\leq \eps\left(\|u_x\|_{H^{j}}^2+\|\tau_x\|_{H^{j-1}}^2\right)+
         C\left(\|v\|_{L^2}^2+\|\psi_t\|_{H^j}^2+\|\psi_x\|_{H^j}^2\right).
\end{equation}
%
%

To this end, we notice that from \eqref{eqn:P} 
we have the identity
\begin{align*}
P&=\left(g(\tilde{U})-g(\bar{U})\right)\psi_x=\left(\int_0^1dg(\bar{U}+\theta v)d\theta\right)v\psi_x.\\
\end{align*}
Using Sobolev embedding then, we can estimate $P$ 
in $H^1$ in a straight forward way, using that $\|v\|_{L^\infty}$ is assumed to be
small (say, at most one).  Indeed, using the above integral representation for $P$ we immediately obtain
\[
\|P\|_{L^2}\lesssim\|v\psi_x\|_{L^2}\lesssim\|v\|_{L^2}\|\psi_x\|_{L^\infty}\lesssim\|v\|_{L^2}\|\psi_x\|_{H^1}
\]
and similarly
\begin{align*}
\|P_x\|_{L^2}&\lesssim\left\|\left(\int_0^1d^2g(\bar{U}+\theta v)d\theta\right)(\bar{U}_x+v_x)v\psi_x\right\|_{L^2}+\left\|\left(\int_0^1dg(\bar{U}+\theta v)d\theta\right)(v\psi_x)_x\right\|_{L^2}\\
&\lesssim\|v\psi_x\|_{L^2}+\|v_x\psi_x\|_{L^2}\|v\|_{L^\infty}+\|(v\psi_x)_x\|_{L^2}\\
&\lesssim\|v\|_{L^2}\|\psi_x\|_{H^1}+\|v_x\|_{L^2}\|\psi_x\|_{H^1}\|v\|_{L^\infty}+\|v_x\|_{L^2}\|\psi_x\|_{H^1}+\|v\|_{L^\infty}\|\psi_{xx}\|_{L^2}\\
&\lesssim\left(\|v\|_{H^1}+\|v\|_{H^1}^2\right)\|\psi_x\|_{H^1}\\
&\lesssim\|v\|_{H^1}\|\psi_x\|_{H^1}.
\end{align*}
From these estimates together with Cauchy-Schwarz it follows that
\begin{align*}
\mathcal{E}_1[v,wP]&\lesssim\|v\|_{L^2}\|P\|_{L^2}+\|v_x\|_{L^2}\|P\|_{L^2}+\|v_x\|_{L^2}\|(wP)_x\|_{L^2}\\
&\lesssim\|v\|_{H^1}^2\|\psi_x\|_{H^1}
\end{align*}
where, again, we have used the fact that $\|\psi_x\|_{H^1}$ is small.  Since we can control the size of $\psi_x$ in $H^1$,
it follows that the $\|v_x\|_{L^2}$ term above can be absorbed in the sense that the above inequality is of
form \eqref{eqn:absorb}.

Using similar arguments, we can express $T(v)$ as
\[
T(v)=g(\tilde{U})-g(\bar{U})-dg(\bar{U})v=\left(\int_0^1\left(\int_0^1d^2g(\bar{U}+\theta s v)ds\right)\theta v~d\theta\right)v.
\]
from which we get the estimates
\[
\|T(v)\|_{L^2}\lesssim\|v\|_{L^\infty}\|v\|_{L^2}\lesssim\|v\|_{H^1}^2 
\]
and, similarly,
\[
\|T(v)_x\|_{L^2}\lesssim\|v\|_{H^1}^2+\|v\|_{H^1}^3\lesssim\|v\|_{H^1},
\]
where we have used the facts that $\|v\|_{H^1}$ is small.  As above, these estimates readily yield
\begin{align*}
\mathcal{E}_1[v,wT]&=\left<v,\Sigma w T\right>+\left<v_x,K wT\right>+\left<v_x,\Sigma(wT)_x\right>\\
&\lesssim\|v\|_{L^2}\|v\|_{H^1}^2+\|v_x\|_{L^2}\|v\|_{H^1}^2
\end{align*}
which again absorbs due to the control over $v$ in $H^1$.
%

To analyze the remaining terms of $\mathcal{N}$, consider the term
$
\left(B(\tilde{U})(\bar{U}_x+v_x)\frac{\psi_x}{1+\psi_x}\right)_x
$
present at the end of \eqref{N}.
Using the representation
$
B(\tilde{U})=\left(\int_0^1dB(\bar{U}+\theta v)d\theta\right)v+B(\bar{U})
$
along with the smallness of $v$ in $H^1$ and $\psi_x$ in $H^1$,
it follows that the associated
contribution to $\mathcal{E}_1[v,w\mathcal{N}]$ can be absorbed so long as the
highest-order
Friedrich's term
\[
\left<v_x,\Sigma\left(w\left(B(\tilde{U})(\bar{U}_x+v_x)\frac{\psi_x}{1+\psi_x}\right)_x\right)_x\right>
\]
can be shown to absorb.  Using integration by parts,
we have
\begin{align*}
\left<v_x,\Sigma\left(w\left(B(\tilde{U})(\bar{U}_x+v_x)\frac{\psi_x}{1+\psi_x}\right)_x\right)_x\right>
    &=-\left<\left(\Sigma v_x\right)_x,w\left(B(\tilde{U})(\bar{U}_x+v_x)\frac{\psi_x}{1+\psi_x}\right)_x\right>,\\
\end{align*}
which absorbs by estimates similar to those previously obtained, using the smallness of $v$ in $H^1$ and $\psi_x$ in $H^2$.
%
%

To estimate the contribution of the final terms of $\mathcal{E}_1[v,w\mathcal{N}]$, associated with $-Q_x$,
first write $Q:=Q_1-Q_2+Q_3$ where
\begin{align*}
Q_1&:=f(\tilde{U})-f(\bar{U})-df(\bar{U})v=\left(\int_0^1\left(\int_0^1d^2f(\bar{U}+\theta sv)ds\right)\theta v~d\theta\right)v\\
Q_2&:=B(\tilde{U})\tilde{U}_x-B(\bar{U})\bar{U}_x-B(\bar{U})v_x\\
Q_3&:=\left(dB(\bar{U})\bar{U}_x\right)v
\end{align*}
and notice that the contribution of $\mathcal{E}_1[v,w(Q_{1}+Q_3)_x]$ absorbs using estimates analogous to those obtained above
for $P$ and $T$.  To illustrate how to handle the contributions of $Q_2$, first notice that by \eqref{solvedderivs} we have
\[
Q_2=\left(\int_0^1dB(\bar{U}+\theta v)d\theta ~v\right)\left(\bar{U}_x+v_x\right)\frac{1}{1+\psi_x}-B(\bar{U})\left(\bar{U}_x+v_x\right)\frac{\psi_x}{1+\psi_x}
\]
which absorbs as above by the smallness of $v$ in $H^1$ and $\psi_x$ in $H^2$.

From the above considerations, then, we immediately have the following lemma.

\begin{lemma}
Under the same hypothesis of Lemma \ref{l:Fried1}, we have the first-order ``Kawashima-type" estimate
\[
\frac{d}{dt}\mathcal{E}_1[v,v]\leq -\theta_1\left(\|u_x\|_{H^1}^2+\|\tau_x\|_{L^2}^2\right)+C\left(\|v\|_{L^2}^2+\|\psi_x\|_{H^1}^2+\|\psi_t\|_{H^1}^2\right)
\]
valid for some constants $\theta_1,C>0$.
\end{lemma}

With this $H^1$ estimate in hand, the analogous $H^m$ estimate follows for any $m\in\mathbb{N}$, as in the statement of Lemma \ref{l:gen1}.
Finally, using one last time the Sobolev inequality $\|g_x\|_{L^2}^2\lesssim\|g_{xx}\|_{L^2}\|g\|_{L^2}$, together with Young's inequality,
we have completed the proof
of the key inequality \eqref{Eineq}, from which the proof of the Proposition \ref{damping} follows.

\section{High-frequency resolvent bounds}

In this appendix, we carry out the high-frequency resolvent
bounds needed for the high-frequency solution operator
bounds of Section \ref{HF}.  To begin, write
$$
L_{\xi} = e^{-i \xi x} L  e^{i \xi x}=
 \hat \partial B \hat \partial   -\hat \partial A +C ,
$$
where $\hat \partial:=(\partial_x+i\xi)$.
Clearly then, the norm
$\|f\|_{\hat H^1}:=\|\hat \partial f\|_{L^2([0,X])}
+\|f\|_{L^2([0,X])}$ is equivalent to the usual
norm $\|f\|_{H^1([0,X])}$ for $\xi\in [-\pi,\pi]$ bounded.
Further, note that, for periodic functions $f$, $g$ on
$[0,X]$, we have the usual integration by parts rule
\be\label{partsform}
\langle f, \hat \partial g\rangle = \langle -\hat \partial f, g\rangle,
\ee
where $\langle \cdot, \cdot \rangle$ as above denotes the standard
$L^2$ complex inner product on $[0,X]$.  The main result of this
appendix is then that for $|\xi|$ bounded away from zero and sufficiently
small the resolvent operator $\left(\lambda-L_{\xi}\right)^{-1}$ is uniformly
$H^1\to H^1$ bounded for $\Re(\lambda)=-\eta<-\theta<0$ for some constant $\theta>0$,
which is the content of the following lemma.

\bl\label{resbd}
Under the derivative condition \eqref{froudebd},
there exist constants $C,R>0$ and a constant $\theta>0$ sufficiently small such that
for $|\lambda|\ge R$ and $\Re \lambda <-\theta$,
\be\label{e:resbd}
\|w\|_{H^1([0,X])}\le C \|(L_\xi-\lambda)w\|_{H^1([0,X])}.
\ee
\el

\begin{proof}
For $\Sigma$, $K$, as defined in the proof of Proposition
\ref{damping}, define the first-order Kawashima-Bloch bilinear form
as
\[
\widehat{\mathcal{E}}_1[v_1,v_2]:=\left<v_1,\Sigma v_1\right>+\left<\hat\partial v_1,Kv_2\right>+\left<\hat\partial v_1,\Sigma\hat\partial v_2\right>
\]
and suppose $w$ is a solution of $\left(\lambda-L_{\xi}\right)w=f$.  Then using the coercivity estimate
of Lemma \ref{l:kaw}, it follows by taking the real part of the equation
\[
\widehat{\mathcal{E}}_1[w,\left(\lambda-L_\xi\right)w]=\widehat{\mathcal{E}}_1[w,f]
\]
and using the equivalence of $\widehat{\mathcal{E}}_1[w,w]\sim\|w\|_{H^1}^2$,
we obtain similarly as in the proof of Proposition \ref{damping}
\be\label{realest}
(\Re \lambda+\tilde \theta) \|w\|_{\hat H^1}^2
+\tilde \theta \|B\hat \partial_x^2 w\|_{L^2}^2
\le C(\|w\|_{L^2}^2+\|f\|_{\hat H^1}^2),
\quad \tilde \theta>0.
\ee
Similarly, defining the analogous first-order Friedrich's-Bloch bilinear form
\[
\widehat{\mathcal{F}}_1[v_1,v_2]:=\left<v_1,\Sigma v_1\right>+\left<\hat\partial v_1,\Sigma\hat\partial v_2\right>
\]
and taking the imaginary part of the equation
\[
\widehat{\mathcal{F}}_1[w,\left(\lambda-L_\xi\right)w]=\widehat{\mathcal{F}}_1[w,f]
\]
we obtain%
\be\label{imagest}
|\Im \lambda| \|w\|_{\hat L^2}^2\le
C(\|w\|_{\hat H^1}^2+
\|B\hat \partial_x^2 w\|_{L^2}^2+
\|f\|_{\hat H^1}^2).
\ee
Summing \eqref{realest} with a sufficiently small multiple
of \eqref{imagest}, we obtain for $\Re \lambda >-\tilde \theta/2$
$$
|\lambda|\|w\|_{\hat H^1}^2\le
C(\|w\|_{\hat H^1}^2+\|f\|_{\hat H^1}^2),
$$
yielding the result for $|\lambda|>2C$
by equivalence of $H^1$ and $\hat H^1$.
\end{proof}


\section{The subcharacteristic condition and Hopf bifurcation}\label{s:hopf}

At equilibrium values $u=\tau^{-1/2}>0$,
the inviscid
version \be\label{relax2}
U_t+f(U)_x=\bp0\\q(U)\ep,
\quad q(U)=1-\tau u^2,
\;
f(U)=\bp -u\\ \frac{1}{2F\tau^2}\ep
\ee
of \eqref{relax} has hyperbolic characteristics equal
to the eigenvalues $\pm \frac{u^3}{\sqrt{F}}$
of $df$, and equilibrium characteristic
$\frac{u^3}{2}$ equal to $\partial_\tau f(\tau, u_*(\tau))$,
where $u_*(\tau):=\tau^{-1/2}$ is defined by $q(\tau, u_*(\tau))=0$.
The subcharacteristic condition, i.e., the condition that the
equilibrium characteristic speed lie between the hyperbolic characteristic
speeds, is therefore
\be\label{subeq}
\frac{u^3}{2}< \frac{u^3}{\sqrt{F}},
\ee
or $F<4$ as stated in Remark \ref{relaxrmk}.

For $2\times 2$ relaxation systems such as the above,
the subcharacteristic condition is exactly the
condition that constant solutions be linearly
stable, as may be readily verified by computing the dispersion
relation using the Fourier transform.
For the full system \eqref{relax} with viscosity $\nu> 0$,
a similar computation, Taylor expanding the dispersion
relation about $\xi=0$, reveals that constant solutions are
stable with respect to {\it low-frequency perturbations}
if and only if the subcharacteristic condition $F<4$ is satisfied.

Next, let us examine the profile ODE
$
c^2 \tau'+ ((2F)^{-1}\tau^{-2})'=
1- \tau (q-c\tau)^2 -c\nu (\tau^{-2}\tau ')'
$
near an equlibrium $u_0=(q-c\tau_0)=\tau_0^{-1/2}>0$, and examine the
circumstances for which Hopf bifurcation occurs.
Linearizing about $\tau\equiv \tau_0$,
and rearranging, we obtain
\be \label{e:linprof}
c\nu \tau_0^{-2}\tau '' + (c^2  -c_s^2) \tau'
+\Big(\frac{u_0^3/2 - c}{u_0/2} \Big) \tau,
\quad
c_s:=\frac{u_0^3}{\sqrt{F}},
\ee
for which the eigenvalues are the roots $\mu$ of
$
\alpha \mu^2 + \beta \mu + \gamma=0,
$
where $\alpha=c\nu \tau^{-2}$,
$\beta = c^2-c_s^2$, and $\gamma= \frac{u_0^3/2 - c}{u_0/2}$.
Considering this as a problem indexed by parameters $u_0$,
$c$, and $q$, we see that Hopf bifurcation occurs when
roots $\mu_j(u_0,c,q)$ cross the imaginary axis as
a conjugate pair, i.e., when $\beta=0$ and $\gamma > 0$.

These translate, using \eqref{subeq} to the {\it Hopf bifurcation conditions}
\be\label{e:hopfcond}
c=c_s= \frac{u_0^3}{\sqrt{F}}
\quad \hbox{\rm and }\;
F>4.
\ee
Experiments of \cite{N1} indicate that bifurcation occurs at
{\it minimum wave speed}, i.e., as $c$ increases through the value $c_s$.
That is, the minimum wave speed among nontrivial periodic waves is
\be\label{minspeed}
c=c_s=\frac{u_0^3}{\sqrt{F}}=
\frac{1}{\sqrt{F \tau_0^3}} ,
\ee
and the minimum value of $F$ for which nontrivial periodic waves
occur is $F>4$.
The frequency at bifurcation is $\omega= \sqrt{\gamma/\alpha}
=
\tau_0^{5/2}\nu^{-1/2} \sqrt{(\sqrt{F}-2)},
$
and the period is $X=\frac{2\pi}{\omega}$.

So prescribing $X$ as we do, we must choose $F>4$, then solve
$\omega=\frac{2\pi}{X}$ to obtain
\be\label{tau0}
\tau_0= \nu^{1/5}\Big(
\frac{4\pi^2}{X^2 (\sqrt{F}-2)}
\Big)^{1/5} .
\ee
Nearby this value and with $c$ near $c_s$, we should find small-amplitude
periodic waves.

\br\label{instabrmk}
\textup{
The above discussion shows in passing that, similarly as observed
in the conservative case in \cite{OZ1}, small-amplitude periodic
waves arising through Hopf bifurcation from constant solutions
are necessarily {\it unstable} as solutions of the time-evolutionary
PDE, since they inherit (a small perturbation of) the necessarily unstable
dispersion relation of the limiting constant solution from
which they bifurcate.
%
On the other hand, in the large-amplitude limit, roll waves might
well be stable.  As observed by Gardner (see \cite{G,OZ1}), this
is determined by stability of the bounding homoclinic wave,
which in the conservative case was known to be unstable.
A good starting point for the study of roll waves, therefore, might
be to determine linearized stability of {\it solitary pulse solutions}
corresponding to homoclinic solutions of the profile ODE.
Evidence for linearized stability of {\it some} viscous roll waves is
given in \cite{N1}, namely, the approximate Dressler waves arising
in the small viscosity limit.
}
\er

\section{Numerical stability investigation}
We conclude by suggesting a number of practical
techniques for the numerical testing of stability.
These can be carried out either in the Eulerian coordinates of \cite{N2}
or in the Lagrangian coordinates of this paper.
As suggested in a more general setting in \cite{B}, several of
the algorithms may be eaily adapted from an existing nonlinear evolution
code.
Comparison of these different methods, and determination of stability
in different regimes, are interesting problems that we hope to carry
out in future work.

\subsection{Method one: the power method}\label{m1}
An easy numerical method to approximate the function
$R(\xi):=\max \Re \sigma(L_\xi)$ determining stability,
with $R(\xi)<0$ for $\xi\ne 0$ corresponding to (D1)
and $R(\xi)\le -\theta \xi^2$ corresponding to (D2).
(Condition (D3) can be verified by an Evans analysis, as
was already done in some cases in \cite{N2} and elsewhere.)

The method is just to approximate numerically the time-evolution
of the linearized equation $ w_t=L_\xi w$ on $[0,X]$
with periodic boundary conditions, which should be a straightforward
adaptation/simplification of the nonlinear code you have already
written to study nonlinear stability with respect to periodic perturbations,
and which should work quite well.  Denote the solution operator
as $e^{L_\xi t}$.  Then, a good approximation is
\be\label{Rapprox}
R(\xi)\approx
T^{1} \log \frac{|e^{L_\xi 2T}f|_{L^2}} {|e^{L_\xi T}f|_{L^2}},
\ee
where $f$ is a square wave pulse centered at $x=X/2$
and $T$ is large, say $T=10$, $T=50$ or $T=100$.
This should be relatively straightforward, and plotting
$R(\xi)$ agains $\xi$ for $\xi\in [-\pi,\pi]$ should
quickly tell stability.
See \cite{BMSZ} for related investigations and discussion.

\subsection{Method two: discretization}\label{m2}
Instead of Evans computations as in \cite{OZ1}
(these involved finding the zero-level-set of a two parameter Evans
function, with reported problematic results)
one could alternatively
proceed from a Bloch decomposition/matrix linear algebra point of view.

That is, one could
discretize $L_\xi$ on $[0,X]$ with periodic boundary conditions
as a large tri-diagonal matrix
\be\label{Tdisc}
T(\xi):=(\Delta+i\xi)B(\Delta+i\xi) -(\Delta+i\xi)A +C,
\ee
acting on vectors
$(U_1, \dots, U_L)$ of sample points, where $U_j\approx U(Xj/L)$,
and virtual point $U_0 \equiv U_L$ (periodicity),
and $\Delta$ is a discrete derivative, for example the forward
difference over $h:=X/L$, treating $[0,X]$ as a torus to wrap
generate needed values $U_j$ for $j\le 0$ or $j>L$.
For each $\xi$, call the fast linear algebra functions in MATLAB
to generate the real part of the largest real part eigenvalue of
$T$ as a function $R(\xi)$.  If $R(\xi)<0$ and $R(\xi)\le -c\xi^2$,
$c>0$, then we have spectral stability- otherwise not.
This should be fast even for $100\times 100$ matrix or so.
The discretization in $\xi$ is over $[-\pi,\pi]$, so also
no problem. $50$ points should suffice.

Note \cite{B} that discretizaton of the linearized operator $L$ is typically
already done for a standard method-of-lines realization of the linearized time
evolution.


\br\label{invpower}
\textup{
In an interesting recent talk by Dwight Barkley \cite{B}, he pointed out that
doing power method for $e^{Lt}$, $t$ small,\footnote{This
is essentially equivalent to the method suggested in Appendix \ref{m1}.}
 with implicit scheme
is something like using inverse power law on $(I-Lt)^{-1}$.
Note that $(I-Lt)^{-1}$ is expected to be compact for $t$ small
in parabolic problems, so this is a
preconditioning step
paralleling the Fredholm theory or Birman--Schwinger
approach on the analytical side.
}
\er

\subsection{Method three: nonlinear evolution}\label{m3}

The simplest test of course is just to run the full nonlinear
problem on a large domain $[-NX,NX]$, $N>>1$, with periodic
boundary conditions and square pulse wave initial conditions
centered at $x=0$.
If the difference between the solution and the unperturbed
periodic wave remains bounded in $L^\infty$, then the wave
is stable, otherwise not.  The experiment should be run only
up to time $T<< NX$ to avoid interactions with the boundary.
This, and the sensitivity of numerical evolution of nonlinear
equations, are the main disadvantages of the method.
The advantage is that this can be converted from
existing nonlinear code for evolution
on a single period $[0,X]$ (easy to change).
%
A variation is to solve the linearized equations
$v_t=Lv:=(\partial_x B\partial_x -\partial_x A + C)v$
numerically, which would be more stable but require
modification (straightforward, however) of the nonlinear
code, changing over to linear.

\subsection{Method four: Evans function computations}\label{m4}

A final approach is to compute the Evans function
$D(\xi, \lambda)$ (straightforward \cite{OZ1,N2}) and plot zero level
sets of $D(\xi, \cdot)$ for varying $\xi$ (harder).
This is not recommended in the basic form just described-
in practice this was time-intensive and
gave poorly resolved results \cite{OZ1}.
%
A somewhat more reasonable variation would be to plot just the level
sets near $(\xi, \lambda)=(0,0)$ (difficult, due to crossing/singularity
at the origin, but contained) to verify
(D2), then use winding number computations for $D(\xi, \cdot)$
to verify (D1).


\end{document}